\documentclass{article}

\usepackage{amsmath,amsfonts,amssymb}
\usepackage{graphicx,subcaption}
\usepackage{braket}
\usepackage{url}
\usepackage{bm}
\usepackage{multirow}
\usepackage{booktabs}
\usepackage{epstopdf}
\usepackage{epsfig}
\usepackage{longtable}
\usepackage{supertabular}
\usepackage{changepage}
\usepackage{tikz}
\usepackage{cases}
\usepackage{setspace}
\usepackage{makecell}
\usepackage{mathtools}
\usepackage{enumerate}
\usepackage{caption}
\usepackage[margin=1in]{geometry}
\usepackage[colorlinks,linkcolor=blue,citecolor=red]{hyperref}
\usepackage{mathrsfs}

\numberwithin{equation}{section}
\allowdisplaybreaks[2]

\usepackage{algorithm}
\usepackage{algorithmicx}
\usepackage{algpseudocode}

\algnewcommand\algorithmicpara{\textbf{Parameters:}}
\algnewcommand\Para{\item[\algorithmicpara]}
\algnewcommand{\AND}{\textbf{and} }
\algnewcommand{\OR}{\textbf{or} }
\algnewcommand{\Break}{\textbf{Break}}

\usepackage{amsthm}

\newtheorem{remark}{Remark}[section]

\newtheorem{lemma}{Lemma}[section]

\newtheorem{proposition}{Proposition}[section]

\newcommand{\ee}{{\rm e}}
\newcommand{\ii}{\mathrm{i}}
\newcommand{\dd}{\mathrm{d}}
\newcommand{\pp}{\partial}

\newcommand{\RR}{\mathbb{R}}
\newcommand{\CFL}{{\rm CFL}}
\newcommand{\Span}{\mathrm{span}}
\newcommand{\Null}{\mathrm{Null}}
\newcommand{\ra}{\rightarrow}

\newcommand{\swH}{\mathcal{H}}
\newcommand{\awH}{\bar{H}}

\DeclareMathOperator*{\argmin}{arg\,min}

\title{Solving Vlasov-Poisson system with an adaptive Hermite spectral method}

\author{
    Sihong Shao\thanks{CAPT, LMAM, and School of Mathematical Sciences, Peking University, Beijing, 100871, China (sihong@math.pku.edu.cn).} \and 
    Yanli Wang\thanks{Beijing Computational Science Research Center, Beijing, 100193, China (ylwang@csrc.ac.cn).} \and 
    Jie Wu\thanks{Center for Data Science, Peking University, Beijing, 100871, China (wujie5@stu.pku.edu.cn).}
}
\date{\today}

\begin{document}

\maketitle

\begin{abstract}
We propose an adaptive Hermite spectral method for the Vlasov-Poisson system based on a recently developed frequency indicator that measures the contribution of the high-order expansion coefficients. Precisely, the symmetrically weighted Hermite basis with a scaling factor is utilized to approximate the distribution function to satisfy the increasing resolution requirement, which, for example, is induced by filamentation. To implement the scaling adjustment, a fast conservative projection operator is constructed in two steps. The first step is to formulate the projection as a constrained optimization problem to preserve key invariants, including mass, momentum, energy, and the $L^2$ norm of the distribution function. The second step is an ODE-based approximation developed to compute the updated expansion coefficients with linear complexity. Numerical experiments with 1D1V and 2D2V settings validate the feasibility and efficiency of this proposed adaptive Hermite method. 

\end{abstract}



\section{Introduction}
The dynamics of a collection of charged particles are usually described by the Vlasov equation coupled with electromagnetic field equations. It arises in a wide range of applications, including plasma physics \cite{Swanson2008}, astrophysics \cite{Palmroth2018, Roytershteyn2018}, cosmology \cite{Rampf2021}, etc. A widely used reduced model is the Vlasov-Poisson system, which describes the evolution of the electron distribution function $f(t,\bm{x},\bm{v})$ in a neutralizing background. In dimensionless form, the governing equations read
\begin{alignat}{2}
    \frac{\pp f}{\pp t} + \bm{v}\cdot\nabla_{\bm x}f - \bm{E}\cdot\nabla_{\bm v}f = 0&,  & \bm x&\in\Omega_{\bm x}, \quad\bm v\in\RR^{d_v}, \label{eq:Vlasov}\\
    \bm{E}=-\nabla_{\bm x}\phi&, \qquad& \bm x&\in\Omega_{\bm x}, \label{eq:ElecField}\\
    -\Delta_{\bm x}\phi=\rho_0-\rho&, & \bm x&\in\Omega_{\bm x}, \label{eq:Poisson}
\end{alignat}
with the spatial domain $\Omega_{\bm x}\subset\RR^{d_x}$ and periodic boundary conditions on $\partial\Omega_{\bm x}$. Here, $\rho(t,\bm x)$ denotes the electron density, and $\rho_0$ is the constant background ion density satisfying the neutrality condition $\int_{\Omega_{\bm x}}(\rho_0-\rho)\,\dd\bm{x}=0$. The Vlasov-Poisson system admits several conserved quantities, including the total mass $M$, momentum $\bm J$, energy $W$, and the $L^2$ norm. The first three are defined in terms of macroscopic variables as
\begin{alignat}{2}
    \rho &= \int_{\RR^{d_v}} f \,\dd\bm{v}, & \qquad M &=\int_{\Omega_{\bm x}}\rho \,\dd\bm{x}, \label{eq:M}\\
    \rho \bm{u} &= \int_{\RR^{d_v}} \bm{v} f \,\dd\bm{v}, &\qquad \bm{J} & =\int_{\Omega_{\bm x}}\rho \bm{u}\,\dd\bm{x}, \label{eq:J}\\
    \mathcal{E} &= \int_{\RR^{d_v}} |\bm{v}|^2 f \,\dd\bm{v}, & \qquad W &= W^E+W^K = \frac{1}{2}\int_{\Omega_{\bm x}} |\bm E|^2\,\dd\bm{x}+\frac{1}{2}\int_{\Omega_{\bm x}}\mathcal{E}\,\dd\bm{x}.\label{eq:W}
\end{alignat}

Various numerical methods have been developed for Vlasov systems, broadly categorized as particle-based and grid-based. Among the former, the Particle-in-Cell (PIC) method \cite{Birdsall2018, Tskhakaya2007, Degond2010} is the most widely utilized due to its simplicity and conservation properties. However, it suffers from inherent statistical noise \cite{Degond2010}. For grid-based methods, the spatial and velocity variables can be discretized separately. For the spatial variable, commonly used discretizations include the finite difference methods \cite{Banks2019, Issan2024}, finite volume methods \cite{Vogman2018, Blaustein2024}, finite element methods \cite{Zaki1988, Juno2018}, Fourier spectral methods \cite{Vencels2015, Camporeale2016, Huang2026}, semi-Lagrangian methods \cite{Cheng1976, Xiong2018, Kormann2019, Liu2025} and dynamic low-rank methods \cite{Cassini2022, Guo2024, Coughlin2024, Einkemmer2025}. For the discretization of velocity space, representative choices include Hermite spectral methods \cite{Holloway1996, Schumer1998, Manzini2017}, Legendre spectral methods \cite{Manzini2016}, and Fourier methods \cite{Klimas1983, Eliasson2007}. 

For collisionless Vlasov systems, the distribution function $f$ may develop increasingly fine-scale structures in phase space during its evolution \cite{Denavit1972, Klimas1987}. This phenomenon, known as the filamentation, leads to highly oscillatory behavior in the velocity variable and poses significant challenges for numerical simulations \cite{Klimas1987, Camporeale2016}. In particular, accurate long-time integration requires sufficiently high resolution in the microscopic velocity space \cite{Issan2024, Issan2025}. In this work, we propose an adaptive Hermite method for the Vlasov-Poisson system. The key component is a scaling adaptive technique that dynamically adjusts the scaling factor based on a recently developed frequency indicator \cite{Xia2021, Chou2023, Shao2025}, thereby refining the resolution and effectively capturing filamentation.

It has long been recognized that the parameters of the Hermite basis play a critical role in the accuracy of Vlasov simulations \cite{Holloway1996, Schumer1998}. This observation has motivated the development of adaptive Hermite methods that dynamically adjust the shift and scaling factor. For example, a projection-error-based selection strategy was proposed in \cite{Camporeale2006}, although it was not implemented in practice. More recently, a physics-based adaptive approach was introduced in which the shift and scaling factor are determined from momentum and temperature, respectively \cite{Pagliantini2023}. A key distinction of the present work is the adaptivity strategy. Specifically, a frequency indicator closely related to the projection error is adopted. Existing approaches typically determine the scaling factor from macroscopic physical quantities, whereas the proposed method is constructed directly from the spectral expansion, thereby characterizing the behavior of the distribution function in a different view.

Another distinction lies in the choice of the Hermite basis function. When utilizing the Hermite spectral method to discretize the microscopic velocity space, which is particularly effective for approximating Maxwellian-type distributions, two main formulations arise depending on the choice of weight function: the asymmetrically weighted (AW) Hermite basis \cite{Vencels2015, Parker2015, Delzanno2015, Camporeale2016, Filbet2022} and the symmetrically weighted (SW) Hermite basis \cite{Manzini2017, Issan2024, Issan2025a, Huang2026}. The AW Hermite basis, originating from Grad’s moment equations \cite{Grad1949}, naturally preserves macroscopic moments such as mass, momentum, and energy. In contrast, the SW Hermite basis can not preserve all these moments simultaneously, but it preserves the $L^2$ norm of the distribution function, a property not guaranteed for the AW formulation and closely related to numerical stability \cite{Holloway1996, Schumer1998, Issan2024}. A unified framework incorporating both bases was proposed in \cite{Kormann2021} by introducing an additional parameter. It is worth noting that existing adaptive Hermite methods for Vlasov equations remain largely restricted to the AW formulation. In the present work, the SW Hermite basis is utilized instead. This choice is motivated by a fundamental limitation of the AW formulation in the scaling adaptivity, which is illustrated through the linear Landau damping problem in Sec. \ref{sec:3-AWSW}, indicating that the SW formulation provides a more suitable framework for the scaling adaptivity.

Despite its advantages, using scaling adaptivity with the SW Hermite basis introduces two challenges in the scaling adjustments. First, the standard $L^2$ projection for the SW Hermite basis does not preserve macroscopic variables and the $L^2$ norm. To address this issue, a constrained optimization formulation is introduced and solved to replace the standard projection, enforcing the conservation of both macroscopic variables and the $L^2$ norm. The second challenge concerns the computational efficiency. While the time evolution of the Vlasov equation typically has linear complexity $O(N)$, a direct implementation of the scaling projection requires matrix-vector multiplications with $O(N^2)$ complexity, which can offset the benefits of adaptivity. To overcome this limitation, an ODE-based approximation of the $L^2$ projection is developed, thereby yielding a fast conservative projection algorithm and ensuring the overall efficiency of the proposed method.

The performance of the proposed method is assessed through a series of benchmark problems in both 1D1V and 2D2V settings. Numerical results demonstrate that the adaptive strategy significantly improves the accuracy of both the potential energy and the distribution function. In particular, the filamentation structures are accurately resolved through scaling adaptivity. The conservation of macroscopic variables and the $L^2$ norm are verified, and computational cost comparisons confirm the efficiency of the proposed adaptive Hermite method.

The rest of this paper is organized as follows. The Fourier-Hermite method and the scaling adaptive algorithm are introduced in Sec.~\ref{sec:3}, with a detailed discussion of the limitations of the AW Hermite basis presented in Sec.~\ref{sec:3-AWSW}. In Sec.~\ref{sec:pro_fast}, the conservation projection and the ODE-based algorithm to perform the scaling adjustment are discussed, with the numerical examples shown in Sec.~\ref{sec:num}. The conclusion and several appendices are shown in Sec.~\ref{sec:con}, App.~\ref{sec:appA}, \ref{sec:appB}, and \ref{sec:appC}.

\section{Numerical scheme}\label{sec:3}
In this section, the general framework of the adaptive Hermite spectral method for the Vlasov equation is presented. Following earlier studies in \cite{Xia2021, Xia2021a, Shao2025}, the adaptive algorithm is implemented as a modification step applied after each time step, and the core idea is to monitor the contribution of high-frequency components in the truncated spectral expansion. In the following, we will first introduce the Fourier-Hermite method for the Vlasov-Poisson system, followed by the detailed adaptive scaling algorithm.

\subsection{Fourier-Hermite method} \label{sec:2-F-H}
The Fourier-Hermite method with the symmetrically weighted (SW) Hermite functions is described in this subsection. For simplicity, only the formulation for the 1D1V Vlasov-Poisson system ($d_x = d_v = 1$) is presented, and the extension to high-dimensional settings is straightforward.

For the velocity domain $\RR$, the SW Hermite functions are defined as
\begin{equation}\label{eq:SW_basis}
    \swH_{k}^{\beta}(v)=\frac{\sqrt{\beta}}{\pi^{1/4}\sqrt{2^{k}k!}}H_{k}(\beta v)\ee^{-\frac{\beta^2 v^{2}}{2}}=\sqrt{\beta}\swH_{k}^{1}(\beta v),\qquad k\geqslant0,
\end{equation}
where $\beta > 0$ is a scaling factor. Here, $H_k$ denotes the Hermite polynomials given by Rodrigues' formula:
\begin{equation}\label{eq:Hermite_poly}
    H_k(v)=(-1)^{k}\ee^{v^2}\frac{\dd^k}{\dd v^k}\left[\ee^{-v^2}\right].
\end{equation}
The SW Hermite functions form a complete orthonormal basis of $L^2(\RR)$ and satisfy the orthogonality relation
\begin{equation}\label{eq:orth_SW}
    \int_{\mathbb{R}}\swH_{l}^{\beta}(v)\swH_{k}^{\beta}(v)\,\dd v = \delta_{kl},
\end{equation}
where $\delta_{kl}$ is the Kronecker Delta symbol. The distribution function $f$ is approximated as
\begin{equation}\label{eq:fN}
    f(t,x,v) \approx f_{N}^{\beta}(t,x,v) = \sum_{k=0}^{N}\hat{f}_{k}^{\beta}(t,x)\swH_{k}^{\beta}(v-\zeta),
\end{equation}
where $N$ is the expansion order and $\zeta$ is the shift parameter. As summarized in \cite{Issan2024, Issan2025a}, the conservation of total mass, energy, and the $L^2$ norm in the semi-discrete system is ensured by restricting the expansion order to even integers and setting the shift to zero, i.e., $N\in 2\mathbb{N}$ and $\zeta=0$. Adopting these conditions, the expansion coefficients are given by
\begin{equation}\label{eq:coe_f}
    \hat{f}_{k}^{\beta}(t,x) = \int_{\RR}\swH_{k}^{\beta}(v)f(t,x,v)\,\dd v,\qquad 0\leqslant k\leqslant N.
\end{equation}
Applying the Galerkin method to the Vlasov equation \eqref{eq:Vlasov} and adopting the properties of the Hermite functions yields the following semi-discrete system:
\begin{equation}\label{eq:semi_disc}
    \frac{\pp \hat{f}_{k}^{\beta}}{\pp t} + \frac{\pp}{\pp x} \left(\frac{1}{\beta}\sqrt{\frac{k}{2}}\hat{f}_{k-1}^{\beta} + \frac{1}{\beta}\sqrt{\frac{k+1}{2}}\hat{f}_{k+1}^{\beta}\right) + E\left(-\beta\sqrt{\frac{k}{2}}\hat{f}_{k-1}^{\beta} + \beta\sqrt{\frac{k+1}{2}}\hat{f}_{k+1}^{\beta}\right) = 0,
\end{equation}
where $\hat{f}_{k}^{\beta} \coloneq 0$ for $k<0$ and $k>N$. Moreover, the macroscopic variables such as the density, momentum, and energy can be computed explicitly through the expansion coefficients, as 
\begin{align}
\label{eq:macro_coe}
    \begin{pmatrix}
        \rho \\ \rho u \\ \mathcal{E}
    \end{pmatrix}
         = \sum_{k=0}^{N} \begin{pmatrix}
        I_{k,0}^{\beta} \\ I_{k,1}^{\beta} \\ I_{k,2}^{\beta}
    \end{pmatrix}\,\hat{f}_{k}^{\beta}(t,x) \eqqcolon\mathcal{I}_\beta^\top\bm{f}^{\beta},
\end{align}
with
\begin{equation}
    \label{eq:I}
     \mathcal{I}_{\beta}^{ } \coloneq \left(I_{k,r}^{\beta}\right)_{0\leqslant k\leqslant N,\, 0\leqslant r\leqslant2}\in\RR^{(N+1)\times3}, \qquad \bm{f}^{\beta}\coloneq (\hat{f}_{0}^\beta, \hat{f}_1^\beta, \dots, \hat{f}_{N}^\beta)^\top.
\end{equation}
The derivations of the coefficients $I_{k,s}^{\beta}$ can be found in \cite{Kormann2021, Issan2025a}, and listed below. 
\begin{equation}\label{eq:coe_I}
    \begin{aligned}
    I_{k,0}^{\beta} &= \begin{cases}
        2^{\frac{1}{2}}\pi^{\frac{1}{4}}\beta^{-\frac{1}{2}}c_k, \quad& 2\mid k,\\
        0, & 2\nmid k,
    \end{cases}
    \qquad\quad
    I_{k,1}^{\beta} = \begin{cases}
        0, & 2\mid k,\\
        2\pi^{\frac{1}{4}}\beta^{-\frac{3}{2}}c_k^{-1}, \quad& 2\nmid k,
    \end{cases}\\
    I_{k,2}^{\beta} &= \begin{cases}
        \beta^{-2}I_{k,0}^{\beta} + 2^{\frac{3}{2}}\pi^{\frac{1}{4}}\beta^{-\frac{5}{2}}kc_k, \quad\ \,& 2\mid k,\\
        0, & 2\nmid k,
    \end{cases} 
    \end{aligned}    
\end{equation}
and $c_k = \sqrt{(k-1)!!/k!!}$ with the convention $(-1)!! = 0!! = 1$.

For the spatial discretization, a Fourier approximation is utilized. Assuming the spatial domain is $\Omega_x = [0, L]$ with the periodic boundary conditions imposed, the Fourier basis functions are 
\begin{equation}\label{eq:Fourier_bas}
    \exp(2\pi\ii l x / L),   \qquad \text{for}~|l|\leqslant N_x/2,
\end{equation}
with $N_x\in 2\mathbb{N}$ denoting the spatial expansion order. The corresponding collocation points are $x_j=jL/N_x$ for $j=0,1,\ldots, N_x-1$. Accordingly, the distribution function $f$ can be approximated as
\begin{equation}\label{eq:f_NNx}
    f_{N,N_x}^{\beta}(t,x,v) = \sum_{l=-N_x/2}^{N_x/2}\,\sum_{k=0}^{N} g_{k,l}^{\beta}(t)\,\ee^{\ii lx\frac{2\pi}{L}} \swH_{k}^{\beta}(v),\qquad
    g_{k,l}^{\beta}(t) = \frac{1}{L}\int_{0}^{L}\ee^{-\ii lx\frac{2\pi}{L}}\hat{f}_{k}^{\beta}(t,x)\,\dd x.
\end{equation}
It is convenient to introduce the distribution functions at the spatial collocation points as 
\begin{equation}
    \hat{f}_{k,j}^{\beta}(t) \coloneq \sum_{l=-N_x/2}^{N_x/2} g_{k,l}^{\beta}(t)\,\ee^{\ii lx_j\frac{2\pi}{L}}, \qquad 0\leqslant j<N_x.\label{eq:f_kj}
\end{equation}
Similarly, the electric field $E$ is approximated as
\begin{equation}
    E_{N_x}(t,x) = \sum_{l=-N_x/2}^{N_x/2} \widehat{E}_{l}(t)\ \ee^{\ii lx\frac{2\pi}{L}},
\end{equation}
and the values of the electric field $E$ at the collocation points are $E_{j}(t) \coloneq E_{N_x}(t,x_j),0\leqslant j<N_x$.
In the 1D1V setting, the Fourier coefficients of the electric field can be solved directly from the Poisson equation \eqref{eq:Poisson} as
\begin{equation}
    \widehat{E}_{0}\equiv0, \qquad \widehat{E}_{l}(t)=\frac{L}{2\pi \ii l}\sum_{k=0}^{N}I_{k,0}^{\beta}\,g_{k,l}^{\beta}(t), \quad l \neq 0.
\end{equation}
Substituting these approximations into the semi-discrete form \eqref{eq:semi_disc} leads to an ODE system
\begin{equation}\label{eq:ODEs}
\begin{aligned}
    \frac{\dd\hat{f}_{k,j}^{\beta}}{\dd t}
    &+ \sum_{l=-N_x/2}^{N_x/2}\frac{2\pi\ii l}{L} \left(\frac{1}{\beta}\sqrt{\frac{k}{2}}g_{k-1,l}^{\beta} + \frac{1}{\beta}\sqrt{\frac{k+1}{2}}g_{k+1,l}^{\beta}\right)\ \ee^{\ii lx_j\frac{2\pi}{L}}\\
    &\qquad + E_{j}\left(-\beta\sqrt{\frac{k}{2}}\hat{f}_{k-1,j}^{\beta} + \beta\sqrt{\frac{k+1}{2}}\hat{f}_{k+1,j}^{\beta}\right) = 0, \qquad 0\leqslant k\leqslant N,\quad0\leqslant j<N_x.
\end{aligned}
\end{equation}
Transformations between $\hat{f}_{k,j}^{\beta}$ and $g_{k,l}^{\beta}$, as well as between $E_j$ and $\widehat{E}_l$, are efficiently implemented using the Fast Fourier Transform (FFT). The ODE system \eqref{eq:ODEs} is solved by a classical fourth-order Runge-Kutta method (RK4) \cite{Kincaid2009}. The computational cost per time step is $O(N N_x + N N_x \log N_x)$.

When the scaling factor $\beta$ is a fixed constant, as in most existing approaches \cite{Schumer1998, Kormann2021, Issan2024, Issan2025a, Huang2026}, the resulting scheme is referred to as the non-adaptive method. In this work, an adaptive Hermite method is constructed to dynamically adjust $\beta$, thereby significantly enhancing performance when solving the Vlasov system.


\subsection{Scaling adaptive algorithm for the microscopic velocity space}\label{sec:3-scale}
Following the idea of the previous works \cite{Xia2021, Xia2021a, Shao2025}, the frequency indicator for the discretization \eqref{eq:f_NNx} and \eqref{eq:f_kj} is defined by
\begin{equation}\label{eq:ind}
    \mathcal{F}[f_{N,N_x}^{\beta}](t) \coloneq \left(\frac{\displaystyle\sum_{k=N-\bar{N}+1}^{N}\sum_{j=0}^{N_x-1}\left(\hat{f}_{k,j}^{\beta}(t)\right)^2}{\displaystyle\sum_{k=0}^{N}\sum_{j=0}^{N_x-1}\left(\hat{f}_{k,j}^{\beta}(t)\right)^2}\right)^{1/2},
\end{equation}
where $\bar{N}=\max(\lfloor N/3\rfloor,2)$. The indicator $\mathcal{F}[f_{N, N_x}^{\beta}]$ measures the contribution of the highest one-third of the retained Hermite modes, and is shown to be closely related to the lower bound of the approximation error \cite{Chou2023}. The scaling adaptive algorithm is designed to control this indicator within a certain range by dynamically adjusting the scaling factor $\beta$.

When simulating the collisionless Vlasov system, the distribution function $f$ may develop increasingly fine-scale structures in phase space and oscillations in velocity space. Accurately resolving this filamentation phenomenon requires progressively higher resolution. For the Hermite basis function \eqref{eq:SW_basis}, the smallest spacing between collocation points scales approximately as $\frac{1}{\beta\sqrt{N}}$ \cite{Shen2011}. Thus, increasing the scaling factor $\beta$ could enhance the resolution in velocity space for a fixed expansion order $N$, which is precisely the goal of introducing the adaptive scaling algorithm.

During the simulation, as the numerical solution becomes more oscillatory, the high-frequency components of its Hermite expansion grow, resulting in an increase of the indicator $\mathcal{F}[f_{N, N_x}^{\beta}]$ defined by \eqref{eq:ind}. When this indicator exceeds a prescribed range, the scaling adaptive algorithm is activated to adjust the value of $\beta$ accordingly. The algorithm is illustrated in detail below. The pseudocode is listed in App.~\ref{sec:appA}.
\begin{itemize}
\item Generally speaking, the scaling factor $\beta$ may take any positive real value. For practical implementation, a logarithmically spaced set of admissible $\beta$ is introduced,
\begin{equation}\label{eq:set_S}
    S \coloneq \{\beta=q_0^m \mid m\in\mathbb{Z},~\beta_{\min}\leqslant q_0^m\leqslant \beta_{\max}\},
\end{equation}
where the parameter $0<q_0<1$ determines the resolution of the grid in $\beta$, and $\beta_{\min}$, $\beta_{\max}$ specify the range.

\item At initialization, the scaling factor $\beta$ is selected by minimizing the frequency indicator over the set $S$. Assuming that the initial distribution function $f_0(x,v)$ is given in explicit form, the indicator associated with each candidate $\beta \in S$ is evaluated. The one yielding the lowest indicator is chosen as the initial scaling factor 
\begin{equation}\label{eq:beta_init}
    \beta_0 = \argmin_{\beta \in S} \mathcal{F}[f_{N, N_x}^{\beta}](0).
\end{equation}
The corresponding indicator is recorded as the initial reference value
\begin{equation}
    \mathcal{F}^{(s)} = \mathcal{F}[f_{N, N_x}^{\beta_0}](0).
\end{equation}
For subsequent scaling adjustments, a reference range $[\mathcal{F}_l^{(s)}, \mathcal{F}_h^{(s)}]$ is introduced, with lower and upper thresholds defined by
\begin{equation}\label{eq:scale_thres}
    \mathcal{F}_l^{(s)} = \eta_l^{(s)} \mathcal{F}^{(s)}, \qquad \mathcal{F}_h^{(s)} = \eta_h^{(s)} \mathcal{F}^{(s)},
\end{equation}
where $\eta_l^{(s)}<1$ and $\eta_h^{(s)}>1$ are prescribed ratio parameters.

\item After each RK4 time step, the frequency indicator of the numerical solution $f_{N,N_x}^{\beta}(t^n,x,v)$ is evaluated. If the indicator $\mathcal{F}[f_{N,N_x}^{\beta}](t^n)$ is outside the union of two intervals
\begin{equation}\label{eq:scale_ind_range}
    R\coloneq(-\infty,\mathcal{F}_0]\cup [\mathcal{F}_l^{(s)}, \mathcal{F}_h^{(s)}],
\end{equation}
where the constant $\mathcal{F}_0$ is a small positive value introduced to prevent spurious adjustment caused by round-off errors, a bidirectional line search in $\beta$ is activated. See Fig.~\ref{fig:0a} for a schematic illustration. Specifically, the two neighboring points of $\beta$ in $S$ are considered,
\begin{equation}\label{eq:beta_de_in}
    \beta_{\text{de}} = \beta q_0,\qquad \beta_{\text{in}} = \beta / q_0.
\end{equation}
The projections of $f_{N,N_x}^{\beta}$ to these scaling factors, together with their frequency indicators, are computed via
\begin{equation}\label{eq:f_de_in}
    f_{N,N_x}^{\beta_{\text{de}}} = \mathcal{T}^{\beta\ra\beta_{\text{de}}}f_{N,N_x}^{\beta}, \qquad
    f_{N,N_x}^{\beta_{\text{in}}} = \mathcal{T}^{\beta\ra\beta_{\text{in}}}f_{N,N_x}^{\beta},
\end{equation}
where the specific realization of the projection operator $\mathcal{T}^{\beta\ra\beta'}$ is deferred to Sec.~\ref{sec:pro_fast}. Among the three candidates $\beta_{\text{de}}$, $\beta$, and $\beta_{\text{in}}$, the one with the smallest indicator value is chosen as the new $\beta$:
\begin{equation}
    \beta \leftarrow \argmin_{\beta'\in\{\beta_{\rm de}, \beta, \beta_{\rm in}\}}\mathcal{F}[f_{N,N_x}^{\beta'}](t^n).
\end{equation}
This procedure is repeated by updating $\beta$ and searching until either the indicator $\mathcal{F}[f_{N,N_x}^{\beta}](t^n)$ belongs to the union $R$ or a discrete local minimum is reached along the set $\beta\in S$. If the scaling factor $\beta$ is modified in this procedure, denoting $\widetilde{\beta}$ as the finally decided scaling factor, the reference indicator is updated accordingly,
\begin{equation}
    \mathcal{F}^{(s)} \leftarrow \mathcal{F}\left[f_{N,N_x}^{\widetilde{\beta}}\right](t^n),
\end{equation}
and the admissible thresholds \eqref{eq:scale_thres} are recomputed.
\end{itemize}

The scaling adaptive algorithm involves several parameters that affect its sensitivity. A detailed sensitivity analysis of these parameters can be found in \cite{Xia2021}. In this study, the parameter $q_0$ is set to $0.999$, yielding a fine logarithmic grid for the scaling factor $\beta$. The admissible range of $\beta$ is specified by $\beta_{\min}=0.1$ and $\beta_{\max}=30$. The ratios defining the reference indicator range are chosen as $\eta_{l}^{(s)}=0.9992$ and $\eta_{h}^{(s)}=1.0008$. Finally, the round-off parameter is set to $\mathcal{F}_0 = 10^{-13}$.

So far, the framework of the scaling adaptive algorithm has been presented. The complete numerical scheme is summarized as Alg.~\ref{alg:compl_scheme} in App.~\ref{sec:appA}.

\begin{figure}
	\centering
	\begin{subfigure}[b]{0.45\linewidth}
		\includegraphics[width=\textwidth]{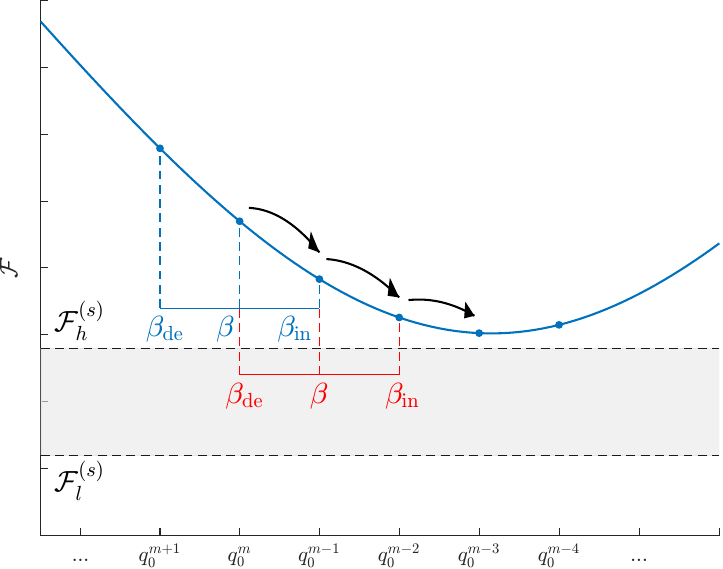}
		\caption*{\footnotesize{$\beta$ ($0<q_0<1$, $m\in\mathbb{N}$)}}
	\end{subfigure}
    \caption{(Adaptive algorithm in Sec.~\ref{sec:3-scale}) Illustration of the scaling adaptive adjustment procedures. When the frequency indicator exits the reference range, the scaling factor $\beta$ is updated via a bidirectional line search on the discrete set $S$ to minimize the indicator.}
	\label{fig:0a}
\end{figure}

\section{Limitation of AW formulation in the scaling adaptivity: the linear Landau damping case}\label{sec:3-AWSW}
An alternative choice of basis functions in Hermite spectral methods is the asymmetrically weighted (AW) Hermite function \cite{Schumer1998, Delzanno2015, Camporeale2016, Filbet2022, Pagliantini2023}. However, when combined with scaling adaptivity, the AW formulation exhibits an inherent restriction on the admissible range of the scaling factor. This limitation is further analyzed in this section with the linear Landau damping problem as a representative example. It also motivates the use of the symmetrically weighted (SW) Hermite basis in this work.



The AW Hermite basis is defined as
\begin{equation}\label{eq:AW_basis}
    \awH_{k}^{\beta}(v)=\frac{\sqrt{\beta}}{(2\pi)^{1/4}\sqrt{2^{k}k!}}H_{k}\left(\frac{\beta v}{\sqrt{2}}\right)\ee^{-\frac{\beta^2 v^{2}}{2}},\qquad k\geqslant0,
\end{equation}
and satisfies the weighted orthogonality relation
\begin{equation}
    \int_{\RR}\awH_k^\beta(v)\awH_l^\beta(v)\,\omega_\beta(v)\,\dd v=\delta_{k,l},\qquad \omega_\beta(v)=\ee^{\frac{\beta^2 v^{2}}{2}}.
\end{equation}
Accordingly, the distribution function is approximated as
\begin{equation}
\label{eq:AP_exp}
    f(t,x,v)\approx\sum_{k=0}^{N}\bar{f}_{k}^{\beta}(t,x)\awH_k^\beta(v), \qquad \bar{f}_{k}^{\beta}(t,x) = \int_{\RR}\awH_k^\beta(v)f(t,x,v)\,\omega_\beta(v)\,\dd v.
\end{equation}
A scaling adaptive algorithm analogous to Sec.~\ref{sec:3-scale} can also be constructed for the AW basis by replacing the SW coefficients $\hat{f}_{k}^{\beta}$ in the indicator \eqref{eq:ind} with the AW coefficients $\bar{f}_{k}^{\beta}$ \cite{Shao2025}. However, a fundamental distinction between the SW and AW scaling adaptive methods lies in their orthogonality and the corresponding approximation spaces. While the SW Hermite basis forms an orthonormal basis of the standard space $L^2(\RR)$, the AW Hermite basis is orthonormal in the weighted space 
\begin{equation}
    L^2_{\omega_\beta}(\RR) \coloneq \left\lbrace u \;\middle|\; \|u\|_{\omega_\beta}^2\coloneq\int_{\RR}|u(v)|^2\,\omega_\beta(v)\,\dd v<\infty\right\rbrace,
\end{equation}
with the weight function $\omega_\beta(v)$ explicitly depending on the scaling factor $\beta$. 
To illustrate the dependence of the approximation space on the scaling parameter $\beta$, assuming $0<\beta_1<\beta_2$, one can immediately obtain
\begin{equation}
    \|u\|_2 \leqslant \|u\|_{\omega_{\beta_1}}\leqslant \|u\|_{\omega_{\beta_2}},
\end{equation}
and consequently
\begin{equation}
    L^2(\RR)\supset L^2_{\omega_{\beta_1}}(\RR)\supset L^2_{\omega_{\beta_2}}(\RR).
\end{equation}
Therefore, as $\beta$ increases, the admissible approximation space for the AW Hermite method becomes strictly smaller. In particular, if $\beta$ is chosen too large, the distribution function $f$ may no longer belong to $L^2_{\omega_\beta}(\RR)$, and the corresponding spectral expansion \eqref{eq:AP_exp} will not converge. This intrinsic restriction limits the flexibility of the scaling adaptivity in the AW Hermite framework for the collisionless Vlasov equation, where increasingly fine structures may require a larger scaling factor $\beta$.

Here, we take the 1D1V linear Landau damping problem for the Vlasov–Poisson system as an illustrative example, with the initial condition 
\begin{equation}
\label{eq:Landau}
    f_0(x,v) = \frac{1+\alpha\cos kx}{\sqrt{2\pi}}\ee^{-\frac{v^2}{2}},
\end{equation}
where $\alpha$ is a small perturbation parameter. One can observe that
\begin{equation}
    \|f_0(x,\cdot)\|_{\omega_\beta}\begin{cases}
        \ <\infty,& \qquad\forall\beta\in(0,\sqrt{2}),\\
        \ =\infty,& \qquad\forall\beta\in(\sqrt{2},\infty),
    \end{cases}\qquad \forall x\in[0,L].
\end{equation}
In particular, for $\beta>\sqrt{2}$, the initial distribution \eqref{eq:Landau} does not belong to $L^2_{\omega_\beta}(\RR)$. Consequently, the AW Hermite expansion is not well-defined in the corresponding weighted space and cannot converge in the $L^2_{\omega_\beta}$ sense. This property persists all the time, which is summarized in Prop.~\ref{prop:sqrt2}.
\begin{proposition}\label{prop:sqrt2}
    Let $f(t,x,v)$ be the solution to the Vlasov-Poisson system with initial condition \eqref{eq:Landau}. Then, for all $x\in[0,L]$ and $t>0$, it holds that 
    \begin{equation}
    \label{eq:L2_fun}
    \|f(t,x,\cdot)\|_{\omega_\beta}\begin{cases}
        \ <\infty,& \qquad\forall\beta\in(0,\sqrt{2}),\\
        \ =\infty,& \qquad\forall\beta\in(\sqrt{2},\infty),
    \end{cases}\qquad \forall x\in[0,L],\quad\forall t>0.
\end{equation}
\end{proposition}
\begin{proof}
    Let $X(s;t,x,v)$, $V(s;t,x,v)$ denote the characteristics associated with the Vlasov equation. Then, it holds that 
    \begin{equation}
        \frac{\dd X}{\dd s}=V(s),\qquad \frac{\dd V}{\dd s}=E(s,X(s)),
    \end{equation}
    with $X(t)=x$, $V(t)=v$. Along with characteristics, the solution is constant, hence
    \begin{equation}
        f(t,x,v)=f_0(X(0;t,x,v),V(0;t,x,v))=\frac{1+\alpha\cos(kX(0))}{\sqrt{2\pi}}\ee^{-\frac{|V(0)|^2}{2}}.
    \end{equation}
    Moreover,
    \begin{equation}
        X(0;t,x,v)=x-\int_0^t V(\tau)\,\dd\tau,\qquad V(0;t,x,v)=v-\int_0^t E(\tau,X(\tau))\,\dd\tau.
    \end{equation}
    For the Landau damping problem, the electric field decays exponentially in time (Thm. 2.6, Eq. (2.19) of \cite{Mouhot2011})
    \begin{equation}
        \|E(t,\cdot)\|_{C^r} \coloneq \sum_{k=0}^{r}\sup_{x}\left|\frac{\pp^k E}{\pp x^k}\right| = O\left(\ee^{-\gamma t}\right),\qquad \text{as }\ t\ra \infty,\qquad \forall r\in\mathbb{N}.
    \end{equation}
    Taking $r=0$ gives
    \begin{equation}
        \sup_{x\in[0,L]}\left|E(t,x)\right|\lesssim\ee^{-\gamma t}.
    \end{equation}
    Define
    \begin{equation}
        U(t)\coloneq\int_0^t E(\tau,X(\tau))\,\dd\tau.
    \end{equation}
    The exponential decay implies
    \begin{equation}
        |U(t)|\leqslant\int_0^t |E(\tau,X(\tau))|\,\dd\tau\leqslant C,\qquad\forall t>0,
    \end{equation}
    for some constant $C$ independent of $t$ and $x$.

    Therefore, the solution $f$ admits an upper bound as 
    \begin{align*}
        f(t,x,v) 
        &\leqslant\frac{1+\alpha}{\sqrt{2\pi}}\ee^{-\frac{(v+U(t))^2}{2}} =\frac{1+\alpha}{\sqrt{2\pi}}\ee^{-\frac{|U(t)|^2}{2}}\ee^{-\frac{v^2}{2}-vU(t)}\leqslant\frac{1+\alpha}{\sqrt{2\pi}}\ee^{-\frac{|U(t)|^2}{2}}\ee^{-\frac{v^2}{2}+C|v|}.
    \end{align*}
    Then, it holds that 
    \begin{equation}
        |f(t,x,v)|^2\ee^{\frac{\beta^2v^2}{2}}\lesssim\exp\left(-\frac{2-\beta^2}{2}v^2+2C|v|\right),
    \end{equation}
    which gives
    \begin{equation}
        \|f(t,x,\cdot)\|_{\omega_\beta}<\infty,\qquad\forall \beta\in(0,\sqrt{2}).
    \end{equation}
    Moreover, the lower bound is given by
    \begin{align*}
        f(t,x,v)
        &\geqslant\frac{1-\alpha}{\sqrt{2\pi}}\ee^{-\frac{(v+U(t))^2}{2}}=\frac{1-\alpha}{\sqrt{2\pi}}\ee^{-\frac{|U(t)|^2}{2}}\ee^{-\frac{v^2}{2}-vU(t)}\geqslant\frac{1-\alpha}{\sqrt{2\pi}}\ee^{-\frac{C^2}{2}}\ee^{-\frac{v^2}{2}-C|v|}.
    \end{align*}
    which gives
    \begin{equation}
        \|f(t,x,\cdot)\|_{\omega_\beta}=\infty,\qquad\forall \beta>\sqrt{2}.
    \end{equation}
    The proof is finished. 
\end{proof}

Prop.~\ref{prop:sqrt2} shows that, for the linear Landau damping problem with the initial condition \eqref{eq:Landau}, the scaling factor in the AW framework must satisfy $\beta<\sqrt{2}$ throughout the whole simulation to ensure the convergence of the Hermite expansion. In contrast, the discussion in Sec.~\ref{sec:3-scale} indicates that the scaling factor $\beta$ should be increased over time to resolve progressively finer structures in velocity space. This conflict might arise in other examples when numerically solving the Vlasov equation with AW Hermite spectral methods, where $\beta$ is intrinsically restricted by an upper bound. Consequently, the scaling adaptivity cannot operate freely, which limits its effectiveness compared with the SW formulation, for which all $\beta\in\RR^+$ are feasible. This provides the primary motivation for adopting the SW Hermite basis in this work.




\section{Fast conservative projection}
\label{sec:pro_fast}
To complete the scaling adaptive Hermite method, the implementation of the projection operator $\mathcal{T}$ in \eqref{eq:f_de_in} is specified in this section. Here, a conservative projection to perform the scaling adjustment is first introduced. The projection is formulated as a constrained optimization problem to enforce the conservation of local mass, momentum, energy, and the $L^2$ norm, which is similar to the strategies developed in Fourier spectral methods to enforce moment conservation \cite{Gamba2009, Zhang2017, Pareschi2022, Cai2024}. Here, the approach is extended to the Hermite spectral method, with an additional quadratic constraint to preserve the $L^2$ norm. 

Then, an ODE-based approximation is proposed to reduce the computational complexity to $O(N)$, following the ideas in \cite{Cai2010a, Hu2020} developed for the AW Hermite method, which contrasts with the $O(N^2)$ complexity of the previous work \cite{Shao2025}. Combining two components together yields a fast conservative projection operator $\mathcal{T}$, thereby completing the construction of the adaptive Hermite method for the Vlasov equations.

\subsection{Projection with moments and \texorpdfstring{$L^2$}{L2} norm conservation}\label{sec:3-proj}
Unlike the scaling projections utilized in the AW Hermite basis \cite{Pagliantini2023, Shao2025}, the standard $L^2$ projection associated with the SW basis preserves neither the macroscopic moments (density, momentum, and energy) nor the $L^2$ norm. Therefore, a carefully designed projection operator $\mathcal{T}$ is required when adjusting the scaling factor $\beta$ in \eqref{eq:f_de_in} to ensure both conservation and stability. This motivates the construction of a modified projection operator described below.

Define the truncated SW Hermite space associated with the scaling factor $\beta$ as 
\begin{equation}
    V_{N}^{\beta} = \Span\left\{\swH_{k}^{\beta}(v),\;0\leqslant k\leqslant N\right\}.
\end{equation}
The projection operator $\mathcal{T}^{\beta\ra\beta'}: V_{N}^{\beta}\ra V_{N}^{\beta'}$ maps a function expanded by the basis functions with scaling factor $\beta$ to its approximation expanded by the basis functions with $\beta'$. Without loss of generalization, the temporal and spatial variables are omitted for simplicity. Denote the preimage and the image functions as 
\begin{align}
    f_{N}^{\beta}(v) = \sum_{k=0}^{N} \hat{f}_{k}^{\beta}\swH_{k}^{\beta}(v),\qquad \mathcal{T}^{\beta\ra\beta'}f_{N}^\beta(v) = \sum_{k=0}^{N} \hat{f}_{k}^{\beta'}\swH_{k}^{\beta'}(v).\label{eq:proj_Tf}
\end{align}
The rest of this subsection aims to determine the new coefficients $\bm{f}^{\beta'} = (\hat{f}_{0}^{\beta'}, \cdots, \hat{f}_{N}^{\beta'})^T$.
\begin{remark}
A natural approach to obtain $\hat{f}_{k}^{\beta'}$ is utilizing the standard $L^2$ projection, denoted as
\begin{align}\label{eq:org_pro}
    \tilde{f}_{k}^{\beta'} &\coloneq \int_{\RR}\swH_{k}^{\beta'}(v)f_{N}^{\beta}(v)\,\dd v,\qquad k = 0,\cdots, N.
\end{align}
However, this $L^2$ projection cannot preserve the density, momentum, energy, or $L^2$ norm of $f_N^{\beta}$.   
\end{remark}
To restore these conservation properties, a constrained optimization problem is proposed, and $\bm{f}^{\beta'}$ is therefore defined as the solution of the following constrained optimization problem
\begin{subequations}
\label{eq:opt_problem}
    \begin{align}
    \min_{\bm{h}\in \RR^{N+1}}\quad &\ \bigl\|\bm{h}-\tilde{\bm{f}}^{\beta'}\bigr\|^2, \label{eq:opt_obj}\\
    \text{s.t.}\quad &\ \mathcal{I}_{\beta'}^\top\bm{h} = \mathcal{I}_\beta^\top\bm{f}^\beta, \label{eq:opt_cons}\\[0.5ex]
    &\ \bigl\|\bm{h}\bigr\|=\bigl\|\bm{f}^\beta\bigr\|, \label{eq:opt_L2}
\end{align}
\end{subequations}
where $\|\bm{h}\|$ denotes the $L^2$ norm of a vector $\bm{h}\in\RR^{N+1}$, $\bm{f}^{\beta}$ is defined in \eqref{eq:I}, and $\tilde{\bm{f}}^{\beta'} = ( \tilde{f}_{0}^{\beta'}, \cdots, \tilde{f}_{N}^{\beta'})$ defined in \eqref{eq:org_pro}. Here, $\mathcal{I}_{\beta}^\top \cdot$ is the matrix defined in \eqref{eq:I}, and the linear constraint \eqref{eq:opt_cons} enforces conservation of the density, momentum, and kinetic energy defined in \eqref{eq:macro_coe}, while the quadratic constraint \eqref{eq:opt_L2} guarantees preservation of the $L^2$ norm.

Moreover, since the electric field $E$ is determined by the density, which is invariant under \eqref{eq:opt_cons}, the total energy defined in \eqref{eq:W} is therefore preserved. The resulting minimizer provides the best approximation of $f_N^\beta$ in $V_N^{\beta'}$ under these conservation constraints. Before introducing the solution to this constrained optimization problem, Lem.~\ref{prop:ITI} is first introduced, whose proof is natural from the definition of the matrix $\mathcal{I}_{\beta}$.
\begin{lemma}
\label{prop:ITI}
    Matrix $\;\mathcal{I}_{\beta'}^\top\mathcal{I}_{\beta'}^{}\in\RR^{3\times3}$ is invertible.
\end{lemma}
Then, the solution to this constrained optimization problem is summarized in Prop.~\ref{prop:opt_solu}.

\begin{proposition}\label{prop:opt_solu}
    Let
    \begin{equation}\label{eq:proj_h0}
        \bm{h}_0 \coloneq \mathcal{I}_{\beta'}^{} \left(\mathcal{I}_{\beta'}^\top\mathcal{I}_{\beta'}^{}\right)^{-1} \mathcal{I}_{\beta}^{\top}\bm{f}^\beta,\qquad
        \mathcal{P} \coloneq I_{N+1} - \mathcal{I}_{\beta'}^{}\left(\mathcal{I}_{\beta'}^{\top}\mathcal{I}_{\beta'}^{}\right)^{-1}\mathcal{I}_{\beta'}^{\top},
    \end{equation}
    where $I_{N+1} \in \RR^{N+1\times N+1}$ is the identity matrix. The minimizer $\bm{f}^{\beta'}$ of the optimization problem \eqref{eq:opt_problem} depends on the consistency condition
    \begin{equation}\label{eq:opt_consis}
        \|\bm{h}_0\|\leqslant\|\bm{f}^\beta\|.
    \end{equation}
The solution is divided into the following three cases:
    \begin{enumerate}

        \item If \eqref{eq:opt_consis} holds and $\mathcal{P}\tilde{\bm{f}}^{\beta'}\neq0$, then we have
        \begin{equation}\label{eq:opt_h4}
            \bm{f}^{\beta'} = \bm{h}_0+\alpha\mathcal{P}\tilde{\bm{f}}^{\beta'},  \qquad  \alpha=\frac{\sqrt{\|\bm{f}^\beta\|^2-\|\bm{h}_0\|^2}}{\big\|\mathcal{P}\tilde{\bm{f}}^{\beta'}\big\|}.
        \end{equation}

        \item If \eqref{eq:opt_consis} holds and $\mathcal{P}\tilde{\bm{f}}^{\beta'}=0$, then the minimizer is not unique. The solution is given by
        \begin{equation}\label{eq:alpha_4r}
            \bm{f}^{\beta'} = \bm{h}_0+\alpha\bm{u}_0, \qquad  \alpha=\frac{\sqrt{\|\bm{f}^\beta\|^2-\|\bm{h}_0\|^2}}{\|\bm{u}_0\|}, \qquad \forall \bm{u}_0\in\Null(\mathcal{I}_{\beta'}^\top).
        \end{equation}
           

        \item If \eqref{eq:opt_consis} does not hold, the constraints optimization problem \eqref{eq:opt_problem} is infeasible.  
\end{enumerate}
\end{proposition}

\begin{proof}
Let us first consider the linear constraint \eqref{eq:opt_cons}. Its general solution set can be written as
\begin{equation}
    \bm{h} = \bm{h}_0+\bm{u},\qquad \forall \bm{u}\in\Null(\mathcal{I}_{\beta'}^\top),
\end{equation}
where a particular solution is $\bm{h}_0$ given in \eqref{eq:proj_h0} and $\bm{h}_0\perp\Null(\mathcal{I}_{\beta'}^{\top})$. $\mathcal{P}$ in \eqref{eq:proj_h0} defines the $L^2$ orthogonal projector onto $\Null(\mathcal{I}_{\beta'}^\top)$. The following orthogonality holds
\begin{align}
\label{eq:orth}
	\bm{h}-\bm{h}_0,\,\mathcal{P}\tilde{\bm{f}}^{\beta'}\in \,\Null(\mathcal{I}_{\beta'}^\top), \qquad \bm{h}_0,\,(I_{N+1}-\mathcal{P})\tilde{\bm{f}}^{\beta'}\in \,\Null(\mathcal{I}_{\beta'}^\top)^\perp.
\end{align}
By the triangle inequality and \eqref{eq:orth}, we can deduce that
\begin{equation}\label{eq:tri_ineq}
\begin{aligned}
    \big\|\bm{h}-\tilde{\bm{f}}^{\beta'}\big\|^2
    &= \big\|\bm{h}-\bm{h}_0-\mathcal{P}\tilde{\bm{f}}^{\beta'}\big\|^2 + \big\|\bm{h}_0+\mathcal{P}\tilde{\bm{f}}^{\beta'}-\tilde{\bm{f}}^{\beta'}\big\|^2  
    \geqslant \bigl(\big\|\bm{h}-\bm{h}_0\big\|-\big\|\mathcal{P}\tilde{\bm{f}}^{\beta'}\big\|\bigr)^2 + \big\|\bm{h}_0+\mathcal{P}\tilde{\bm{f}}^{\beta'}-\tilde{\bm{f}}^{\beta'}\big\|^2.
\end{aligned}
\end{equation}
The equality in \eqref{eq:tri_ineq} holds if and only if $\bm{h}-\bm{h}_0$ is collinear with $\mathcal{P}\tilde{\bm{f}}^{\beta'}$. Thus, the minimizer has the form
\begin{equation}
\label{eq:opt_h5}
    \bm{h} = \bm{h}_0+\alpha \mathcal{P}\tilde{\bm{f}}^{\beta'},\qquad \alpha\in\RR.
\end{equation}
Combining \eqref{eq:opt_h5} with the quadratic constraint \eqref{eq:opt_L2} yields
\begin{equation}
\label{eq:opt_h6}
    \bigl\|\bm{h}_0\bigr\|^2+\alpha^2 \bigl\|\mathcal{P}\tilde{\bm{f}}^{\beta'}\bigr\|^2 = \bigl\|\bm{f}^{\beta}\bigr\|^2.
\end{equation}
Then the three cases of the solution are proved successively. 
\begin{enumerate}
    \item The consistency condition \eqref{eq:opt_consis} holds and   $\mathcal{P}\tilde{\bm{f}}^{\beta'}\ne0$. With \eqref{eq:opt_h6}, it is easy to deduce that 
    \begin{equation}
        \label{eq:coe_alpha}
        \alpha=\frac{\sqrt{\|\bm{f}^\beta\|^2-\|\bm{h}_0\|^2}}{\big\|\mathcal{P}\tilde{\bm{f}}^{\beta'}\big\|},
    \end{equation}
    and the first case is proved. 
    \item The consistency condition \eqref{eq:opt_consis} holds and  $\mathcal{P}\tilde{\bm{f}}^{\beta'}=0$. In this case, \eqref{eq:opt_h5} vanishes. The only constraint is $\bm{h}-\bm{h}_0 \in \Null(\mathcal{I}_{\beta'}^\top)$ as indicated in \eqref{eq:orth}. Consequently, the minimizer is not unique. To construct a feasible solution, select an arbitrarily nonzero vector $\bm{u}_0\in\Null(\mathcal{I}_{\beta'}^\top)$. The solution is then written in the form
    \begin{equation}
        \bm{h} = \bm{h}_0+\alpha \bm{u}_0,\qquad \alpha\in\RR.
    \end{equation}
    Similarly, imposing the quadratic constraint \eqref{eq:opt_L2} yields
    \begin{equation}
        \alpha=\frac{\sqrt{\|\bm{f}^\beta\|^2-\|\bm{h}_0\|^2}}{\|\bm{u}_0\|}.
    \end{equation}
    \begin{remark}
    From \eqref{eq:coe_I}, it is easy to verify that $\bm{u}_0$ can be chosen independently of $\beta'$. For example, $   \bm{u}_0=\left(\sqrt{3}/(2\sqrt{2}),0,-\sqrt{3},0,1,0,0,...,0\right)^{\top}$. Though the case $\mathcal{P}\tilde{\bm{f}}^{\beta'}=0$ is discussed, for all the numerical experiments conducted in this work, case of $\|\mathcal{P}\tilde{\bm{f}}^{\beta'}\|\geqslant10^{-14}$ has never been observed. 
    \end{remark}
    \item The consistency condition \eqref{eq:opt_consis} does not hold. In this case, we cannot find a $\alpha \in \RR$ satisfying \eqref{eq:opt_h6}. This also means that the macroscopic moments and $L^2$ norm cannot be conserved simultaneously. Either \eqref{eq:opt_cons} or \eqref{eq:opt_L2} has to be relaxed to make the problem feasible. In this work, priority is given to the conservation of the macroscopic moments. The quadratic constraint \eqref{eq:opt_L2} is therefore removed, and the minimization is performed over the affine space defined by \eqref{eq:opt_cons}. In this case, the solution can be directly given by
    \begin{equation}\label{eq:opt_h3}
        \bm{h} = \bm{h}_0+\mathcal{P}\tilde{\bm{f}}^{\beta'}.
    \end{equation}
    \begin{remark}
        The eigenvalue analysis presented in App.~\ref{sec:appB} indicates that \eqref{eq:opt_consis} does not hold universally. However, for all the numerical experiments conducted in this work, violations of \eqref{eq:opt_consis} have not been observed. 
    \end{remark}
\end{enumerate}

\end{proof}

\begin{remark}
    If, alternatively, preservation of the $L^2$ norm is preferred over moment conservation, the linear constraint \eqref{eq:opt_cons} may be removed. The minimizer of \eqref{eq:opt_problem} under the sole constraint \eqref{eq:opt_L2} is then given by a simple rescaling
    \begin{equation}
        \bm{h}=\frac{\left\|\bm{f}^\beta\right\|}{\big\|\tilde{\bm{f}}^{\beta'}\big\|}\tilde{\bm{f}}^{\beta'}.
    \end{equation}
\end{remark}

The solution of the constrained optimization problem \eqref{eq:opt_problem}, which determines the projected coefficients $\hat{f}^{\beta'}_k$ in \eqref{eq:proj_Tf}, has now been characterized for all cases. The remaining task is the efficient evaluation of the $L^2$ projection coefficients $\tilde{\bm{f}}^{\beta'}$ in \eqref{eq:org_pro}.

\subsection{Fast implementation via ODE-based approximation}\label{sec:fast_pro}
In general, the $L^2$ projection between two SW Hermite expansions with different scaling factors is evaluated via matrix-vector multiplication, leading to a computational complexity of $O(N^{d_v+1})$. A similar cost arises for the AW Hermite basis \cite{Pagliantini2023}. However, incorporating such an $O(N^{d_v+1})$ procedure into the scaling adaptive algorithm will increase the order of complexity and become computationally prohibitive in practice. To address this issue, an ODE-based approach is proposed to approximate the $L^2$ projection with a reduced complexity of $O(N^{d_v})$.

Without loss of generality, only the case $d_v = 1$ is considered, and it is natural to extend to higher dimensions. For $s\in[0,1]$, define an intermediate scaling factor and the associated coefficients as
\begin{align}
    \tilde{h}_l(s)\coloneq \int_{\RR}\swH_l^{b(s)}(v)f_N^\beta(v)\,\dd v, \qquad 
    b(s) = \beta\left(\frac{\beta'}{\beta}\right)^s, \qquad 0\leqslant l\leqslant N. \label{eq:fs}
\end{align}
It is straightforward to verify that $\tilde{h}_l(0)=\hat{f}_l^\beta$ and $\tilde{h}_l(1)=\tilde{f}_l^{\beta'}$. Therefore, $\tilde{h}_l(s)$ defines a continuous path connecting the two sets of Hermite coefficients $\bm{f}^{\beta}$ and $\tilde{\bm{f}}^{\beta'}$, which is summarized in Prop.~\ref{prop:proj_ODE} and the proof is provided in App.~\ref{sec:appC}.

\begin{proposition}\label{prop:proj_ODE}
    The intermediate coefficients $\tilde{h}_l(s)$ defined in \eqref{eq:fs} satisfy
    \begin{equation}\label{eq:proj_ODE}
        \frac{\dd\tilde{h}_l}{\dd s}=\frac{\ln(\beta'/\beta)}{2}\left(\sqrt{l(l-1)}\tilde{h}_{l-2}-\sqrt{(l+1)(l+2)}\tilde{h}_{l+2}\right),\qquad 0\leqslant l\leqslant N,
    \end{equation}
    with $\tilde{h}_{-1}=\tilde{h}_{-2}=\tilde{h}_{N+1}=\tilde{h}_{N+2}\equiv0$. Or it can be rewritten as
    \begin{equation}\label{eq:proj_Lf}
        \frac{\dd\tilde{\bm h}}{\dd s} = \mu \mathcal{L}\tilde{\bm h},
    \end{equation}
    where $\mu = \frac{\ln(\beta'/\beta)}{2}$. $\mathcal{L}=(L_{i,j})\in\RR^{(N+1)\times(N+1)}$ is a skew-symmetric matrix with elements
    \begin{equation}
        L_{j+2,j}=-L_{j,j+2}=\sqrt{(j+1)(j+2)},\qquad j=0,\dots,N-2,
    \end{equation}
    and $L_{i,j}=0$ otherwise.
\end{proposition}

With this formulation, the $L^2$ projection is changed into an initial value problem, and the objective is to compute $\tilde{\bm h}(1)$. Here, the coefficient matrix $\mu\mathcal{L}$ is independent of $s$, and the system admits an analytic solution
\begin{equation}
\label{eq:sol_ODE}
    \tilde{\bm h}(s) = \ee^{s\mu \mathcal{L}}\tilde{\bm h}(0),\qquad s\in[0,1].
\end{equation}
Various approaches can be utilized to approximate $\ee^{\mu\mathcal{L}}\tilde{\bm h}(0)$ with $O(N)$ complexity \cite{Moler1978}. In this work, a simple and efficient implementation is obtained by directly integrating \eqref{eq:proj_Lf} using a fourth-order Runge-Kutta (RK4) method with step size $\Delta s=1/N_b$. Since $\tilde{h}_l$ only couples with $\tilde{h}_{l-2}$ and $\tilde{h}_{l+2}$ as in \eqref{eq:proj_Lf}, each time step requires $O(N)$ operations, resulting in an overall computational complexity of $O(N N_b)$. To determine $N_b$, noting that 
\begin{equation}\label{eq:mu}
    \mu=\frac{\ln(\beta'/\beta)}{2}=\pm\frac{\ln q_0}{2},
\end{equation}
and the parameter $q_0$ is chosen close to $1$ in the scaling adaptive algorithm, the coefficient $\mu$ is small, and hence the entries of the matrix $\mu\mathcal{L}$ are uniformly small. Consequently, a high-accuracy approximation of $\tilde{\bm{f}}^{\beta'}$ can be achieved with a relatively small number of time steps $N_b$. In the numerical experiments presented in this work, $N_b=10$ is sufficient, as no discernible improvement is observed when a larger $N_b$ is adopted.

Moreover, unlike the AW formulation, for which the RK4 discretization is unconditionally stable \cite{Cai2010a}, the stability condition of the RK4 discretization of \eqref{eq:proj_Lf} requires further consideration. To this end, we examine the spectrum of the matrix $\mathcal{L}$. Since $\mathcal{L}$ is skew-symmetric, all its eigenvalues are purely imaginary. Prop.~\ref{prop:rhoL} provides an upper bound for its spectral radius.
\begin{proposition}\label{prop:rhoL}
    The spectral radius of $\mathcal{L}$ satisfies
    \begin{equation}
        \rho(\mathcal{L}) \leqslant 2N.
    \end{equation}
\end{proposition}
\begin{proof}
    Applying the Gershgorin circle theorem (see Prop. 5.12 in \cite{Serre2010}) to the matrix $\mathcal{L}$, each eigenvalue $\lambda_j$ of $\mathcal{L}$ lies within at least one Gershgorin disc centered at zero with radius given by the sum of the absolute values of the off-diagonal entries in the corresponding row. From the structure of $\mathcal{L}$, each row contains at most two nonzero entries, whose magnitudes are bounded by $\sqrt{(i+1)(i+2)}\leqslant N$. Therefore,
    \begin{equation*}
        |\lambda_j| \leqslant \max_{i=0,\dots,N-2}\sqrt{i(i-1)}+\sqrt{(i+1)(i+2)} \leqslant 2N, \qquad \forall j,
    \end{equation*}
    which implies $\rho(\mathcal{L})\leqslant 2N$.
\end{proof}

The classical fourth-order Runge-Kutta (RK4) method applied to \eqref{eq:proj_Lf} can be written as
\begin{equation}
    \tilde{\bm h}^{n+1} =\left[I_{N+1} + \Delta s\,\mu \mathcal{L} + \frac{1}{2}(\Delta s\,\mu \mathcal{L})^2 + \frac{1}{6}(\Delta s\,\mu \mathcal{L})^3 + \frac{1}{24}(\Delta s\,\mu \mathcal{L})^4\right] \tilde{\bm h}^n \eqqcolon  R(\Delta s\,\mu \mathcal{L}) \tilde{\bm h}^n.
\end{equation}
For $L^2$ stability, it is required that
\begin{equation}
    \rho\bigl(R(\Delta s\,\mu \mathcal{L})\bigr)\leqslant 1.
\end{equation}
Since $\mathcal{L}$ has purely imaginary eigenvalues $\lambda_j=\ii\omega_j$, the stability condition of the RK4 method along the imaginary axis reads \cite{Butcher2016}
\begin{equation}
    \Delta s\,|\mu \omega_j|\leqslant2\sqrt{2},\qquad \forall j.
\end{equation}
With \eqref{eq:mu} and Prop.~\ref{prop:rhoL}, this condition is satisfied provided that
\begin{equation}\label{eq:stab}
    \Delta s \leqslant \frac{2\sqrt{2}}{N|\ln q_0|}.
\end{equation}
This indicates that for a fixed step size $\Delta s$, the parameter $q_0$ should be chosen sufficiently close to $1$ when the expansion order $N$ is large. 
\begin{remark}
For the 1D1V numerical experiments in this work, the parameters are set as $\Delta s = 0.1, q_0 = 0.999$, 
which satisfy the stability condition \eqref{eq:stab} for all $N\leqslant \frac{20\sqrt{2}}{|\ln 0.999|}\approx28270$.
\end{remark}

Combining the ODE-based approximation with the solution to the constrained problem \eqref{eq:opt_problem} for each spatial index $j=0,...,N_x-1$, a fast and conservative implementation of the projection operator $\mathcal{T}^{\beta\ra\beta'}$ is obtained. The complete procedure is summarized in Alg.~\ref{alg:proj}.

\begin{algorithm}
\caption{Implementation of the fast conservative projection.}
\label{alg:proj}
\begin{algorithmic}[1]
    \Require Local distribution coefficients $\bm{f}^{\beta}$, current scaling factor $\beta$, new scaling factor $\beta'$;
    \Ensure Local distribution coefficients after projection $\bm{f}^{\beta'}$;
    \Para Steps of the ODE approximation $N_b = 10$.
    \State \label{sch:12}Solve the ODE \eqref{eq:proj_Lf} using an RK4 method with $\Delta s=1/N_b$ to $s=1$, which gives the approximated $L^2$ projection $\tilde{\bm{f}}^{\beta'}$;
    \State \label{sch:13}Compute $\bm{h}_0$ and $\mathcal{P}\tilde{\bm{f}}^{\beta'}$ by \eqref{eq:proj_h0};
    \If{$\|\bm{h}_0\|\leqslant\|\bm{f}^\beta\|$ \AND $\|\mathcal{P}\tilde{\bm{f}}^{\beta'}\|\geqslant10^{-14}$}
    \State $\bm{f}^{\beta'}=\bm{h}_0+\alpha \mathcal{P}\tilde{\bm{f}}^{\beta'}$ with $\alpha$ given by \eqref{eq:opt_h4};
    \ElsIf{$\|\bm{h}_0\|\leqslant\|\bm{f}^\beta\|$ \AND $\|\mathcal{P}\tilde{\bm{f}}^{\beta'}\|<10^{-14}$}
    \State $\bm{f}^{\beta'}=\bm{h}_0+\alpha \bm{u}_0$ with $\bm{u}_0$ and $\alpha$ given by \eqref{eq:alpha_4r};
    \Else 
    \State $\bm{f}^{\beta'}=\bm{h}_0+\mathcal{P}\tilde{\bm{f}}^{\beta'}$;
    \EndIf
    
\end{algorithmic}
\end{algorithm}

The computational complexity of the ODE approximation in Step~\ref{sch:12} is $O(N N_b)$, while the matrix–vector multiplication in Step~\ref{sch:13} can be evaluated within $O(N)$ operations. Consequently, the overall projection procedure scales as $O(N N_b N_x)$. Since $N_b$ is a small fixed constant, the projection maintains linear complexity with respect to the Hermite order $N$. An outline of the complete adaptive scheme is provided in App.~\ref{sec:appA}.


\section{Numerical results}
\label{sec:num}
In this section, a series of numerical experiments is studied to validate the accuracy, efficiency, and conservation of the proposed adaptive Hermite spectral method. The test cases include 1D1V and 2D2V problems, which demonstrate the effectiveness of the adaptive algorithm for Vlasov simulation.

All simulations are performed on a single Intel Xeon Platinum 8358 processor. For 1D1V problems, a single-thread implementation is utilized, while for 2D2V problems, OpenMP parallelization is employed with 64 threads to accelerate the computation. The time step is determined according to the CFL condition specified in \cite{Shao2025}.

\subsection{1D1V examples}
We first consider several classical 1D1V benchmark problems, including linear Landau damping, nonlinear Landau damping, two-stream instability, and bump-on-tail instability. These problems exhibit various phase-space structures. Thus, the ability of the proposed scaling adaptive method to capture the filamentation will be examined. 

\subsubsection{Linear Landau damping}\label{sec:Eg1-1}
The performance of the scaling adaptive scheme is first investigated for the linear Landau damping problem with the initial condition
\begin{equation}\label{eq:exp1-Landau}
    f_0(x,v)=\frac{1+\alpha\cos(kx)}{\sqrt{2\pi}}\exp\left(-\frac{v^2}{2}\right),
\end{equation}
where the perturbation amplitude $\alpha=0.01$, the wave number $k=0.5$, and $x\in[0,4\pi]$. In this test, the spatial expansion order of the Fourier method is set to $N_x=32$, and the CFL number is $\CFL = 0.8$. The parameters utilized in the scaling adaptive method are listed in Tab.~\ref{tab:1-1-1}.

\begin{table}
    \caption{Parameters utilized in the scaling adaptive algorithm.}
    \label{tab:1-1-1}
    \centering
    \def\arraystretch{1.3}
    {\footnotesize
    \begin{tabular}{c|cccccc}
    parameter & $\beta_{\min}$ & $\beta_{\max}$ & $q_0$ & $\eta_{l}^{(s)}$ & $\eta_{h}^{(s)}$ & $\mathcal{F}_{0}$\\ \hline
    usage & \eqref{eq:set_S} & \eqref{eq:set_S} & \eqref{eq:set_S} & \eqref{eq:scale_thres} & \eqref{eq:scale_thres} & \eqref{eq:scale_ind_range}\\
    value & 0.1 & 30 & 0.999 & 0.9992 & 1.0008 & $10^{-13}$
    \end{tabular}
    }
\end{table}

\begin{figure}
    \centering
    \begin{subfigure}[b]{0.4\linewidth}
		\includegraphics[width=\textwidth]{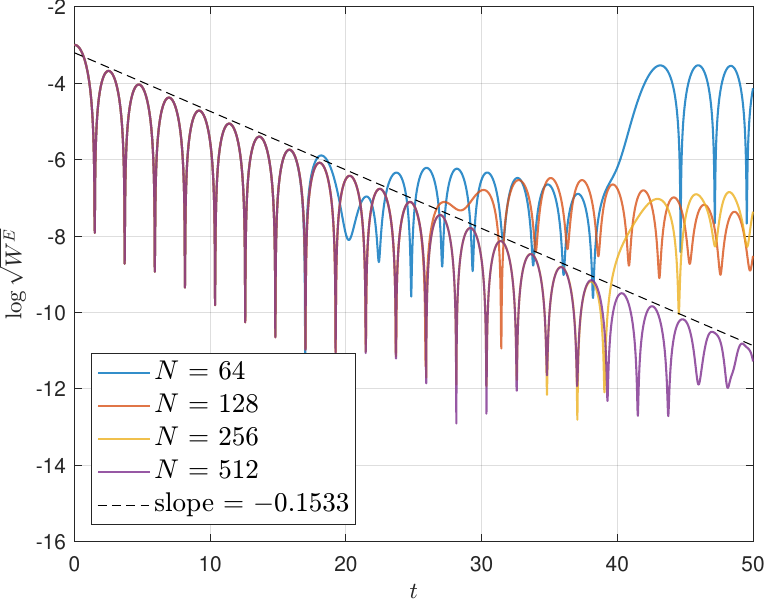}
		\caption{$W^E$, NA}
		\label{fig:1-1-1a}
	\end{subfigure}
	\qquad
	\begin{subfigure}[b]{0.4\linewidth}
		\includegraphics[width=\textwidth]{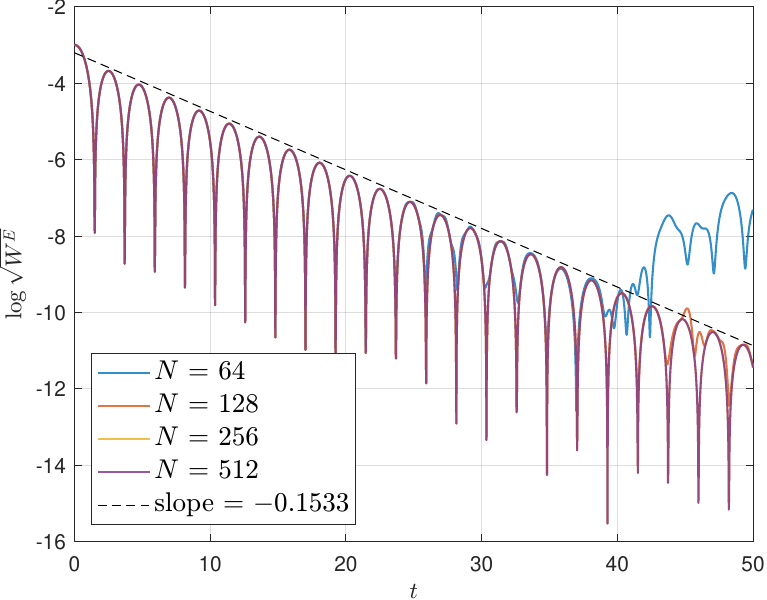}
		\caption{$W^E$, SA}
		\label{fig:1-1-1b}
	\end{subfigure}

	\caption{(1D1V linear Landau damping in Sec.~\ref{sec:Eg1-1}) Comparison between the scaling adaptive and non-adaptive methods with different expansion orders $N$. (a) Evolution of the potential energy $W^E$ by the non-adaptive method (NA). (b) Evolution of the potential energy $W^E$ by the scaling adaptive method (SA). }
	\label{fig:1-1-1}
\end{figure}

According to the dispersion equation, for this linear Landau problem, the electric energy is decaying with a fixed damping rate \cite{Canosa1973} as \eqref{eq:damping_rate}. Moreover, as discussed in \cite{Schumer1998}, the electric energy exhibits a recurrence phenomenon due to the fixed resolution of the Hermite collocation points. The recurrence time $t_{\mathrm{recur}}$ is proportional to the square root of the Hermite expansion order $N$ as in \eqref{eq:damping_rate}.
\begin{equation}
    \label{eq:damping_rate}
    S\left(\log\sqrt{W^E}\right) \approx -0.1533, \qquad t_{\mathrm{recur}}\sim\frac{\pi\sqrt{N}\beta}{k}.
\end{equation}

First, the numerical results for the non-adaptive and scaling adaptive method are shown in Fig.~\ref{fig:1-1-1}, where the expansion numbers are chosen as $N = 64, 128, 256$, and $512$. For the non-adaptive method, the scaling factor is fixed as $\beta = 1$, while the scaling adaptive method is expected to increase the scaling factor $\beta$ with a fixed expansion order $N$ as shown in Fig.~\ref{fig:1-1-2a}, thereby enhancing resolution and delaying recurrence. The results in Fig. \ref{fig:1-1-1} demonstrate that $\beta$ increases gradually over time, leading to a substantial delay of recurrence. For instance, when $N=64$, the scaling adaptive method maintains accurate potential energy up to approximately $t\approx35$, whereas the non-adaptive method exhibits recurrence at approximately $t\approx18$ for $N=64$ and $t\approx26$ for $N=128$. This demonstrates that the scaling adaptive algorithm achieves higher effective resolution without increasing the expansion order.

\begin{figure}
    \centering
	\begin{subfigure}[b]{0.4\linewidth}
		\includegraphics[width=\textwidth]{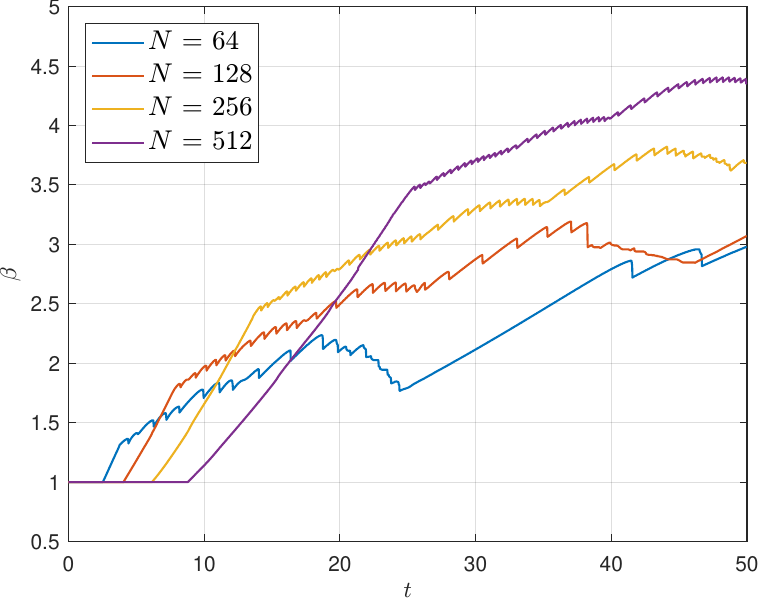}
		\caption{Scaling factor $\beta$}
		\label{fig:1-1-2a}
	\end{subfigure}
    \qquad
    \begin{subfigure}[b]{0.41\linewidth}
		\includegraphics[width=\textwidth]{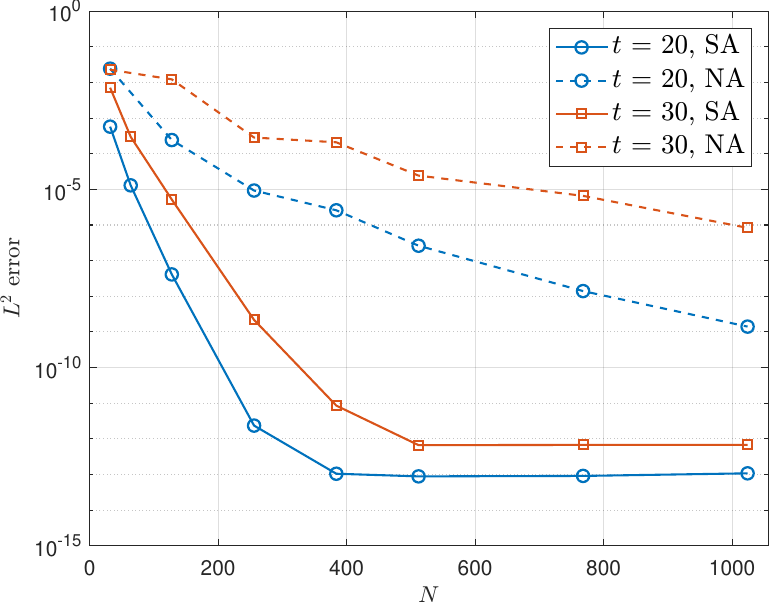}
		\caption{$L^2$ error of $f$}
		\label{fig:1-1-2b}
	\end{subfigure}
	
	\caption{(1D1V linear Landau damping in Sec.~\ref{sec:Eg1-1}) (a) Evolution of the scaling factor $\beta$. (b) $L^2$ error of the scaling adaptive and non-adaptive methods with different expansion orders $N$ at $t=20$ and $30$.}
	\label{fig:1-1-2}
\end{figure}

	

To further assess the performance of the two methods, the variation of the $L^2$ errors of $f$ with respect to the expansion order $N$ at $t = 20$ and $30$ is shown in Fig.~\ref{fig:1-1-2}. The reference solution is computed by the non-adaptive method with $N=5120$. As the distribution function develops increasingly fine structures, the advantage of adaptivity becomes more evident. Specifically, the scaling adaptive method with $N=256$ attains a smaller error than the non-adaptive method with $N=1024$. The results indicate that the scaling adaptive method achieves significantly lower errors and exhibits a faster convergence rate, consistent with the observations reported in \cite{Wang2025, Hu2026}. 

The average computational time per time step for the scaling adaptive and non-adaptive methods is listed in Tab.~\ref{tab:1-1-2}. As discussed in Sec.~\ref{sec:3}, the complexity of the scaling projection introduced by the scaling adaptive method is of the same order as that of the Runge-Kutta time integration. Thus, the total CPU time of the scaling adaptive method increases approximately linearly with respect to $N$. In practice, the adaptive adjustment accounts for about one-fourth to one-third of the total computational cost. Considering the significant gain in accuracy relative to the additional computational overhead, these results demonstrate the effectiveness and efficiency of the scaling adaptive method.

\begin{table}[hptb]
    \caption{(1D1V linear Landau damping in Sec.~\ref{sec:Eg1-1}) Average CPU time per time step (in seconds). Here, $T_{\rm non}$ and $T_{\rm adap}$ refer to the CPU time of the non-adaptive and scaling adaptive methods, respectively. $T_{\rm ind}$ refers to the CPU time spent on the adaptive adjustment step, which is a component of $T_{\rm adap}$.}
    \label{tab:1-1-2}
    \centering
    \def\arraystretch{1.3}
    {\footnotesize
    \begin{tabular}{l|cccc}
     & $T_{\rm non} $ & $T_{\rm adap}$ & $T_{\rm ind}$ & $T_{\rm ind}/T_{\rm adap}$ \\ \hline
    $N=64$ & 4.21E$-$04 & 6.61E$-$04 & 2.29E$-$04 & 34.62\% \\
    $N=128$ & 8.48E$-$04 & 1.20E$-$03 & 3.43E$-$04 & 28.53\% \\
    $N=256$ & 1.60E$-$03 & 2.28E$-$03 & 6.41E$-$04 & 28.11\% \\
    $N=512$ & 3.22E$-$03 & 4.41E$-$03 & 1.14E$-$03 & 25.92\%
    \end{tabular}
    }
\end{table}

\begin{figure}
	\centering
	\begin{subfigure}[b]{0.32\linewidth}
		\includegraphics[width=\textwidth]{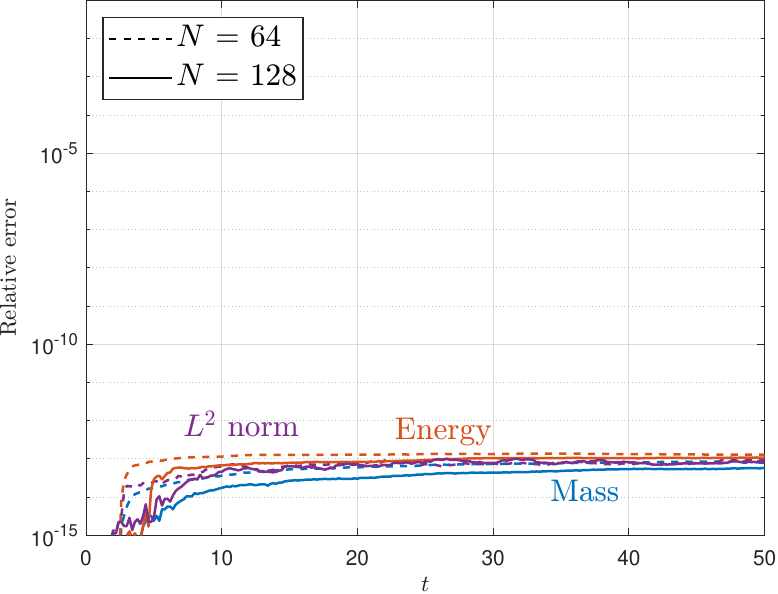}
		\caption{Conservative}
		\label{fig:1-1-3a}
	\end{subfigure}
	\hfill
	\begin{subfigure}[b]{0.32\linewidth}
		\includegraphics[width=\textwidth]{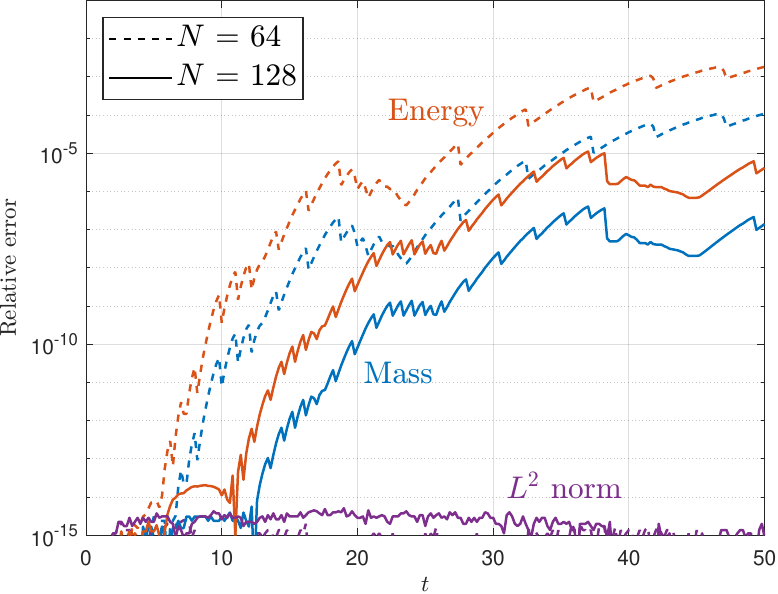}
		\caption{First non-conservative method}
		\label{fig:1-1-3b}
	\end{subfigure}
	\hfill
	\begin{subfigure}[b]{0.32\linewidth}
		\includegraphics[width=\textwidth]{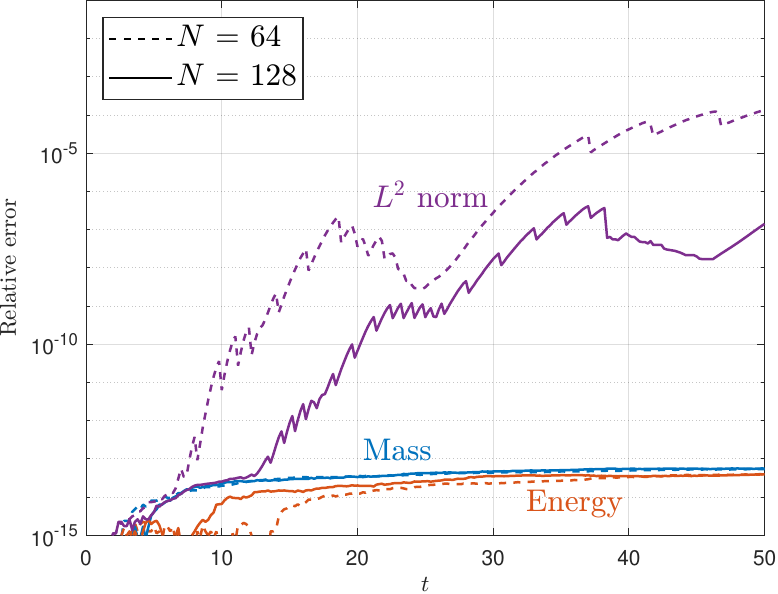}
		\caption{Second non-conservative method}
		\label{fig:1-1-3c}
	\end{subfigure}	
	\caption{(1D1V linear Landau damping in Sec.~\ref{sec:Eg1-1})  Evolution of the relative error for the total mass, energy, and $L^2$ norm with $N = 64$, and $N = 128$. Here, the solid lines are the numerical results with $N = 128$, while the dashed lines are those with $N = 64$. The blue lines are the error evolution of the total mass, the red lines are those of the total energy, while the purple lines are those of the $L^2$ norm. (a) Conservative projection utilized in the scaling adaptive method, preserving both moments and the $L^2$ norm. (b) First non-conservative method. (c) Second non-conservative method.}
	\label{fig:1-1-3}
\end{figure}

\paragraph{Comparison between conservative and non-conservative projections}
To examine the conservative properties of the scaling adaptive method, the numerical results obtained with the conservative projection proposed in Sec.~\ref{sec:pro_fast} are compared with those from two other non-conservative numerical schemes. For the first non-conservative scheme, the expansion coefficients $\tilde{\bm{f}}^{\beta'}$ obtained by the direct $L^2$ projection as \eqref{eq:org_pro} are utilized. Without the optimization problem \eqref{eq:opt_problem}, this method can not preserve the exact conservation of the macroscopic variables or the $L^2$ norm. Nevertheless, due to the anti-symmetric structure of the ODE system \eqref{eq:proj_ODE}, the resulting error in $L^2$ norm is typically small, although the error of the macroscopic variables may accumulate over time. For the second non-conservative method, the optimization problem \eqref{eq:opt_problem} with only the linear moment constraint \eqref{eq:opt_cons} is solved, while the quadratic $L^2$ constraint \eqref{eq:opt_L2} is omitted. This method conserves exactly the prescribed macroscopic variables, such as the density and energy, but does not preserve the $L^2$ norm.

The time evolution of the relative errors in total mass, energy, and the $L^2$ norm for the three methods, with $N=64$ and $128$, is shown in Fig.~\ref{fig:1-1-3}. To suppress the error due to time discretization, the time step size is fixed as $\Delta t=0.002$ in those tests. Here, the total momentum is omitted, since the problem is symmetric with respect to $v=0$, and all three methods preserve the total momentum. Fig.~\ref{fig:1-1-3} indicates that for the conservation method, the relative error of the total mass, energy, and $L^2$ remains quite small at the level of $10^{-13}$ all the time, while that of the total mass and energy is growing rapidly for the first non-conservative method, and that of the $L^2$ norm is quickly growing for the second non-conservative method. Moreover, the numerical behavior for $N = 64$ and $N = 128$ is quite similar. These results confirm the conservative properties of the proposed scaling adaptive method.

	



\subsubsection{Nonlinear Landau damping}\label{sec:Eg1-2}
In this section, the nonlinear Landau damping problem is considered, where the initial condition has the same form as \eqref{eq:exp1-Landau} but with $\alpha=0.5$. The spatial Fourier expansion order is set to $N_x=256$, and the CFL number is $\CFL = 0.8$. The parameters of the scaling adaptive method are the same as those in Tab.~\ref{tab:1-1-1}.

\begin{figure}
    \centering
    \begin{subfigure}[b]{0.4\linewidth}
		\includegraphics[width=\textwidth]{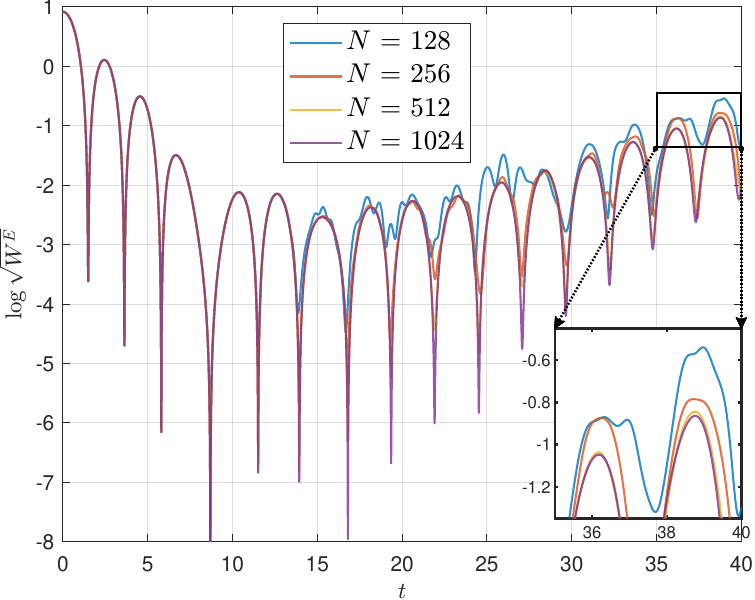}
		\caption{$W^E$, NA}
		\label{fig:1-2-1a}
	\end{subfigure}
    \qquad
	\begin{subfigure}[b]{0.4\linewidth}
		\includegraphics[width=\textwidth]{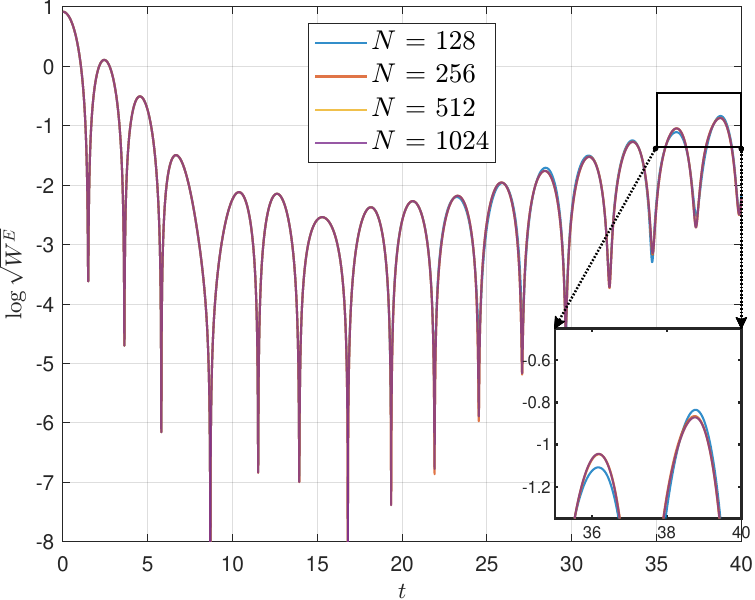}
		\caption{$W^E$, SA}
		\label{fig:1-2-1b}
	\end{subfigure}
	\caption{(1D1V nonlinear Landau damping in Sec.~\ref{sec:Eg1-2}) Comparison between the scaling adaptive and non-adaptive methods with different expansion orders $N$. (a) Evolution of the potential energy $W^E$ by the non-adaptive method. (b) Evolution of the potential energy $W^E$ by the scaling adaptive method.}
	\label{fig:1-2-1}
\end{figure}

The evolution of the electric field obtained by the non-adaptive and scaling adaptive method is plotted in Fig.~\ref{fig:1-2-1a} and \ref{fig:1-2-1b}, respectively. It shows that with low expansion order $N$, the non-adaptive method fails to present an accurate result at the end stage. A mild distinction can be observed even between $N=512$ and $1024$. Whereas for the scaling adaptive method, the result of $N=256$ coincides with that of the higher order. 

For the non-adaptive method, the scaling factor is fixed as $\beta=1$, while the evolution of $\beta$ for the scaling adaptive method is shown in Fig.~\ref{fig:1-2-1c}, which shows a persistent increasing trend over time. This validates the analysis presented in Sec.~\ref{sec:3-scale}, where the proposed scaling adaptive method could capture the filamentation by continuously enlarging $\beta$. Fig.~\ref{fig:1-2-1d} presents the conservation error of the scaling adaptive algorithm with $N=128$. The overall error remains below the order of $10^{-13}$, confirming the conservation properties of the scheme.

To examine the ability to capture the filamentation, the distribution function $f$ obtained by the scaling and non-adaptive method, at $t=20$ and $40$, is presented in Fig.~\ref{fig:1-2-2}. It can be observed that at $t=20$, the non-adaptive method with $N=1024$ is barely able to capture mildly oscillatory structures. However, by $t=40$, it exhibits distinct non-physical oscillations. In contrast, the scaling adaptive method, even with $N=512$, is capable of resolving clear and well-structured filamentation at $t=40$.

The computational time for the non-adaptive and scaling adaptive methods is listed in Tab.~\ref{tab:1-2-1}, where the time spent on obtaining the scaling adaptive adjustment grows approximately linearly with respect to $N$, and its proportion in the total time-stepping cost remains relatively small. Taking both accuracy and computational cost into account, the scaling adaptive method demonstrates superior efficiency compared to the non-adaptive approach.

\begin{figure}
    \centering
	\begin{subfigure}[b]{0.4\linewidth}
		\includegraphics[width=\textwidth]{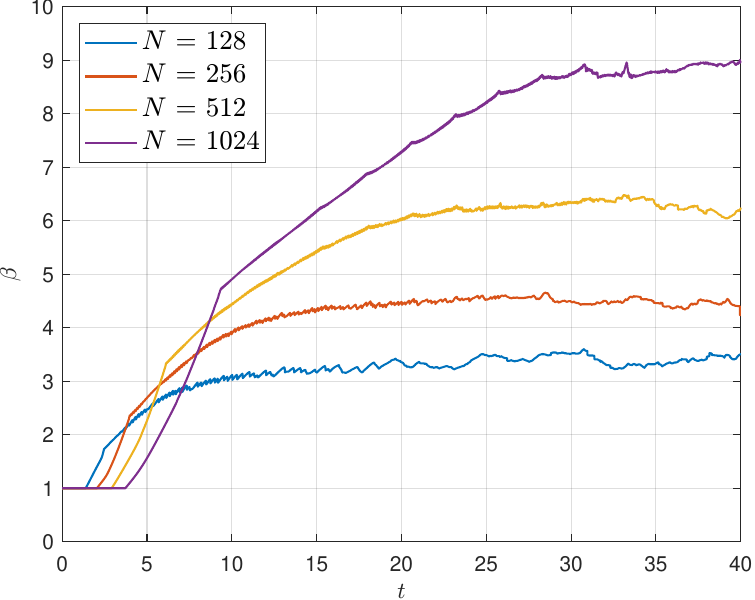}
		\caption{Scaling factor $\beta$}
		\label{fig:1-2-1c}
	\end{subfigure}
    \qquad
    \begin{subfigure}[b]{0.415\linewidth}
		\includegraphics[width=\textwidth]{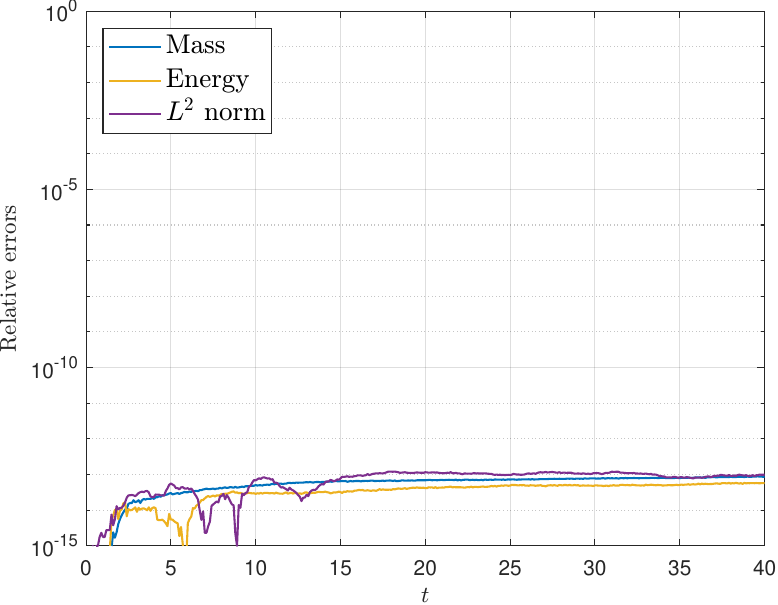}
		\caption{Errors of the conserved variables}
		\label{fig:1-2-1d}
	\end{subfigure}
	
	\caption{(1D1V nonlinear Landau damping in Sec.~\ref{sec:Eg1-2}) Results of the scaling adaptive method with different expansion orders $N$. (a) Evolution of the scaling factor $\beta$. (b) Evolution of the errors in total mass, energy, and the $L^2$ norm, computed by the scaling adaptive method with $N=128$.}
	\label{fig:1-2-3}
\end{figure}

\begin{figure}
    \centering
    \begin{subfigure}[b]{0.32\linewidth}
		\includegraphics[width=\textwidth]{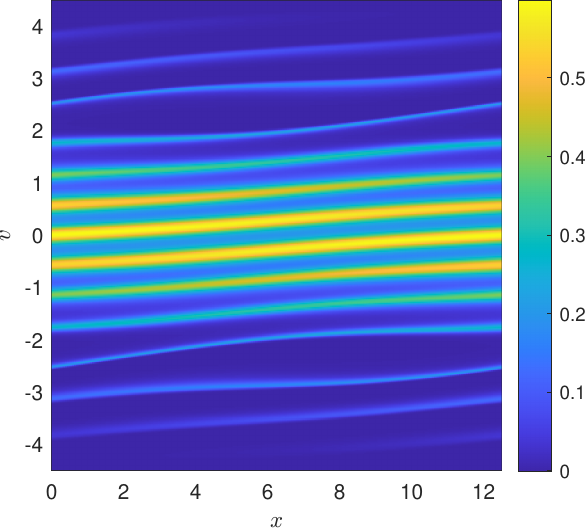}
		\caption{$N=512$, SA, $t=20$}
	\end{subfigure}
    \hfill
    \begin{subfigure}[b]{0.32\linewidth}
		\includegraphics[width=\textwidth]{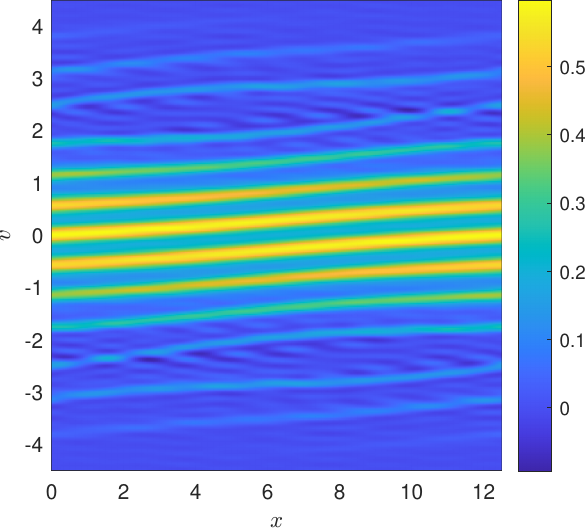}
		\caption{$N=512$, NA, $t=20$}
	\end{subfigure}
    \hfill
    \begin{subfigure}[b]{0.32\linewidth}
		\includegraphics[width=\textwidth]{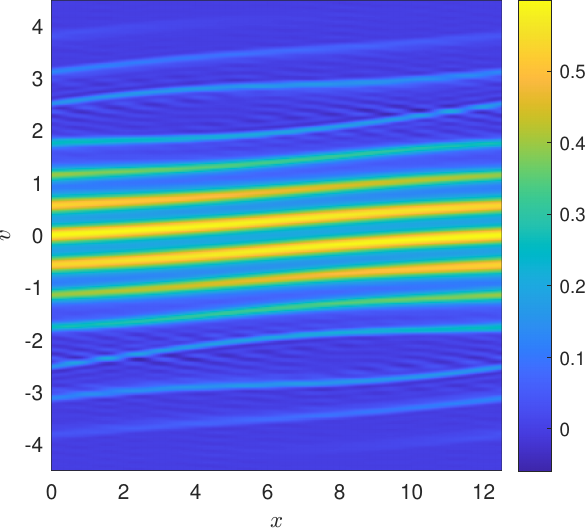}
		\caption{$N=1024$, NA, $t=20$}
	\end{subfigure}
    \\ \medskip
    \begin{subfigure}[b]{0.32\linewidth}
		\includegraphics[width=\textwidth]{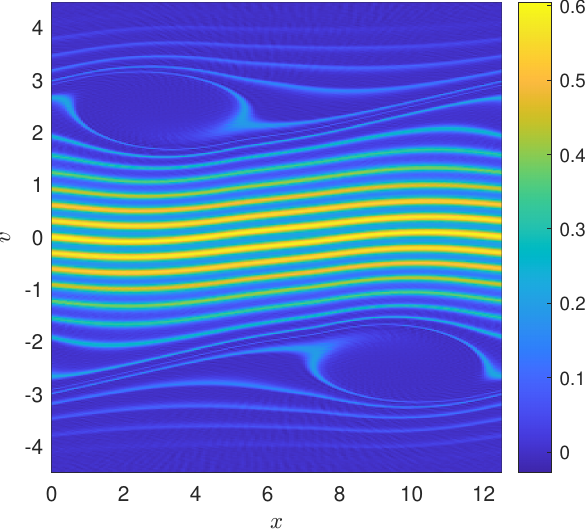}
		\caption{$N=512$, SA, $t=40$}
	\end{subfigure}
	\hfill
    \begin{subfigure}[b]{0.32\linewidth}
		\includegraphics[width=\textwidth]{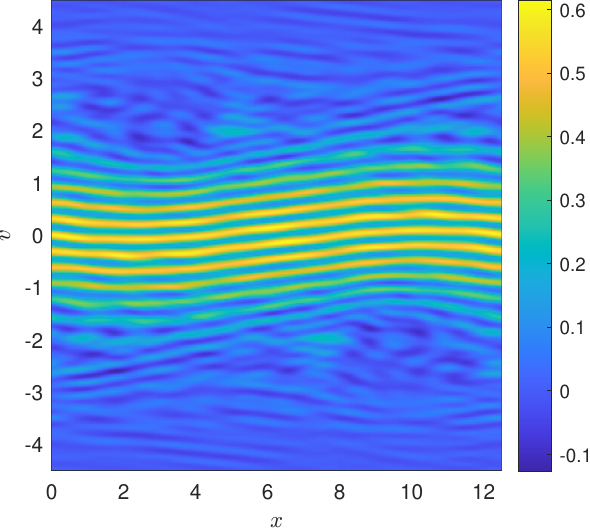}
		\caption{$N=512$, NA, $t=40$}
	\end{subfigure}
    \hfill
	\begin{subfigure}[b]{0.32\linewidth}
		\includegraphics[width=\textwidth]{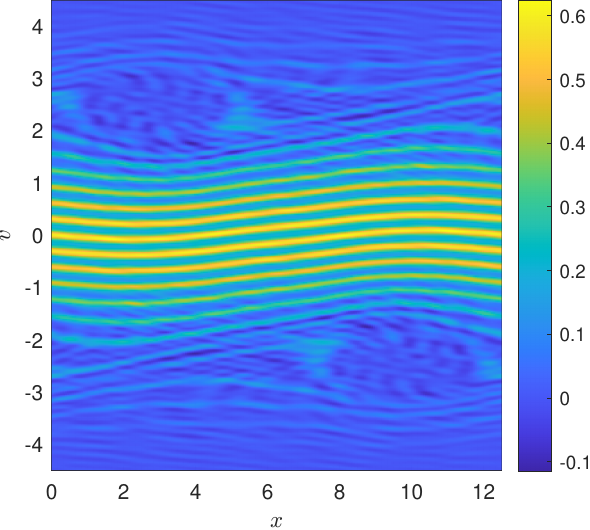}
		\caption{$N=1024$, NA, $t=40$}
	\end{subfigure}	
	\caption{(1D1V nonlinear Landau damping in Sec.~\ref{sec:Eg1-2}) Comparison of the distribution function $f$ between the scaling adaptive and non-adaptive method with different expansion orders. (a--d) $t=20$. (e--h) $t=40$.}
	\label{fig:1-2-2}
\end{figure}

\begin{table}[hptb]
    \caption{(1D1V nonlinear Landau damping in Sec.~\ref{sec:Eg1-2}) Average CPU time per time step (in seconds). Here, $T_{\rm non}$ and $T_{\rm adap}$ refer to the CPU time of the non-adaptive and scaling adaptive methods, respectively. $T_{\rm ind}$ refers to the CPU time spent on the adaptive adjustment step, which is a component of $T_{\rm adap}$.}
    \label{tab:1-2-1}
    \centering
    \def\arraystretch{1.3}
    {\footnotesize
    \begin{tabular}{l|cccc}
     & $T_{\rm non} $ & $T_{\rm adap}$ & $T_{\rm ind}$ & $T_{\rm ind}/T_{\rm adap}$ \\ \hline
    $N=128$ & 8.13E$-$03 & 1.00E$-$02 & 1.85E$-$03 & 18.43\% \\
    $N=256$ & 1.63E$-$02 & 2.04E$-$02 & 3.98E$-$03 & 19.50\% \\
    $N=512$ & 3.34E$-$02 & 4.33E$-$02 & 9.52E$-$03 & 21.98\% \\
    $N=1024$ & 6.30E$-$02 & 8.50E$-$02 & 2.12E$-$02 & 24.92\%
    \end{tabular}
    }
\end{table}

\subsubsection{Two-stream instability}\label{sec:Eg1-3}
In this subsection, the two-stream instability problem with the initial condition \eqref{eq:ini_two_stream} is considered. 
\begin{equation}
\label{eq:ini_two_stream}
    f_0(x,v)=\frac{1+\alpha\cos(kx)}{2\sqrt{2\pi\theta}}\left[\exp\left(-\frac{(v-u)^2}{2\theta}\right)+\exp\left(-\frac{(v+u)^2}{2\theta}\right)\right], \qquad x \in [0, 4\pi],
\end{equation}
where $u=1$, $\theta=0.25$, $\alpha=0.001$, and $k=0.5$. The spatial expansion order is set as $N_x=192$, and the CFL number is $\CFL =0.8$. The parameters for the scaling adaptive method are listed in Tab.~\ref{tab:1-1-1}. For the non-adaptive method, the scaling factor is fixed as $\beta=1/\sqrt{\theta}=2$.


The evolution of the electric field energy obtained by the adaptive and non-adaptive methods is shown in Fig.~\ref{fig:1-3-1a}. In this case, even the non-adaptive method with relatively low expansion order accurately captures the potential energy $W^E$, and the results from all configurations overlap.

The evolution of the scaling factor $\beta$ and the errors of the conserved quantities are presented in Fig.~\ref{fig:1-3-1b} and Fig.~\ref{fig:1-3-1c}, respectively. The scaling factor increases continuously over time. Meanwhile, the relative errors in total mass, energy, and the $L^2$ norm remain below $10^{-13}$, indicating that the scaling adaptive method preserves conservation and maintains $L^2$ stability.

The distribution function $f$ obtained by both methods is compared in Fig.~\ref{fig:1-3-2}. The non-adaptive method with $N=256$ and $512$ exhibits noticeable oscillations near the boundary, which are alleviated when $N$ increases to $1024$. In contrast, the scaling adaptive method resolves the structure clearly with only $N=256$.

Tab.~\ref{tab:1-3-2} reports the average computational time per step of the scaling adaptive and non-adaptive methods. The proportion of time spent on the scaling adjustment remains below $20\%$ in all cases. Consequently, the scaling adaptive method with a lower expansion order $N$ requires significantly less computational time than the non-adaptive method with a higher $N$, while achieving superior accuracy. This demonstrates the efficiency of the proposed adaptive method.

\begin{figure}
	\centering
	\begin{subfigure}[b]{0.32\linewidth}
		\includegraphics[width=\textwidth]{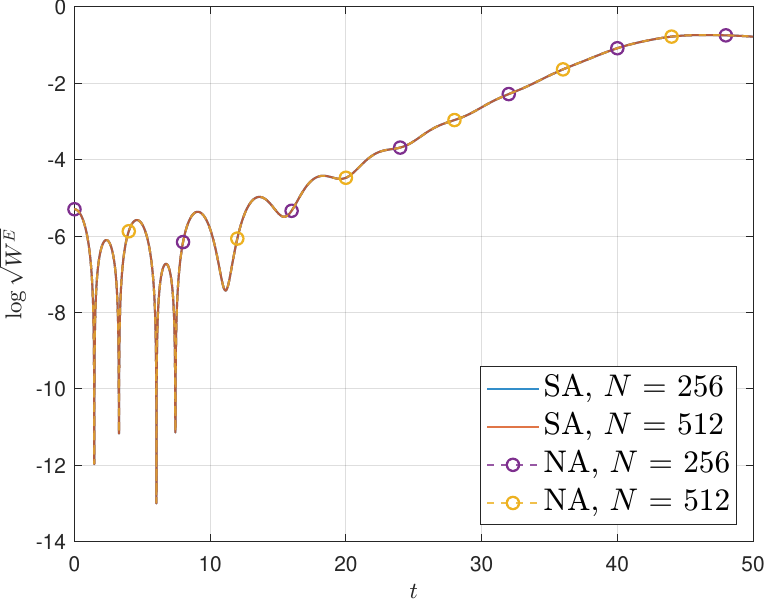}
		\caption{$W^E$}
		\label{fig:1-3-1a}
	\end{subfigure}
	\hfill
	\begin{subfigure}[b]{0.317\linewidth}
		\includegraphics[width=\textwidth]{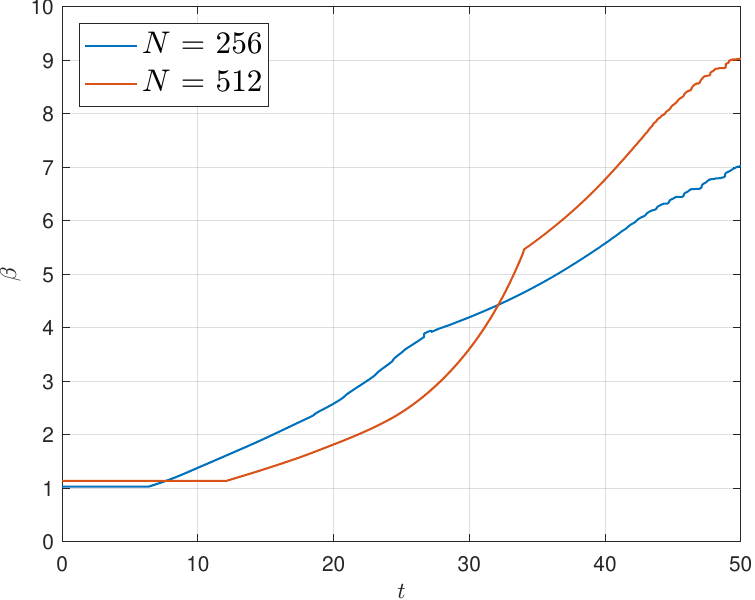}
		\caption{Scaling factor $\beta$}
		\label{fig:1-3-1b}
	\end{subfigure}
    \hfill
    \begin{subfigure}[b]{0.327\linewidth}
		\includegraphics[width=\textwidth]{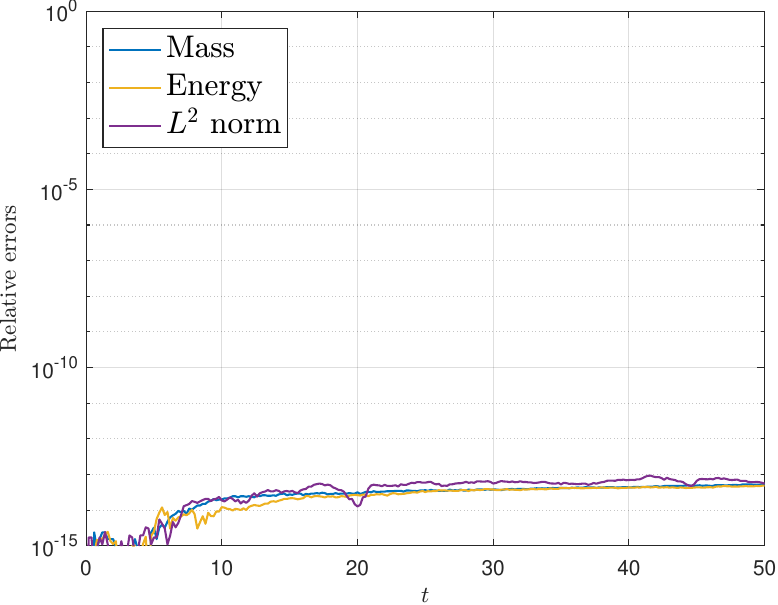}
		\caption{Errors in conserved quantities}
		\label{fig:1-3-1c}
	\end{subfigure}	
	
	\caption{(1D1V two-stream instability in Sec.~\ref{sec:Eg1-3}) Comparison between the scaling adaptive and non-adaptive methods with different expansion orders $N$. (a) Evolution of the potential energy $W^E$ by both methods. (b) Evolution of the scaling factor $\beta$ of the scaling adaptive method. (c) Evolution of the errors in total mass, energy, and the $L^2$ norm, computed by the scaling adaptive method with $N=128$, $\Delta t = 0.001$.}
	\label{fig:1-3-1}
\end{figure}

\begin{figure}
    \centering
    \begin{subfigure}[b]{0.4\linewidth}
		\includegraphics[width=\textwidth]{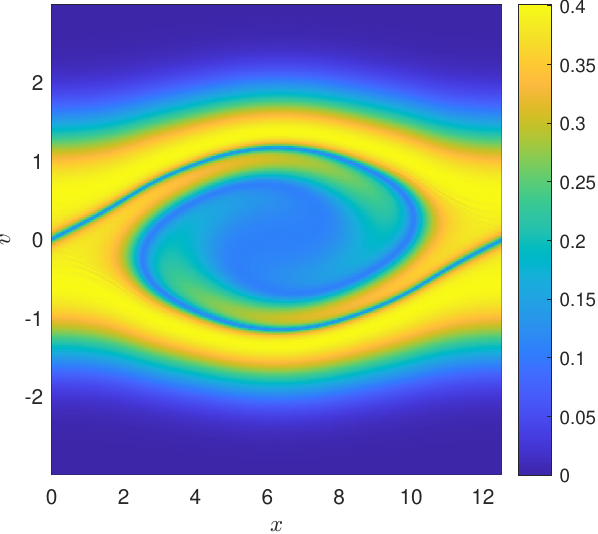}
		\caption{$N=256$, SA}
	\end{subfigure}
	\qquad
    \begin{subfigure}[b]{0.4\linewidth}
		\includegraphics[width=\textwidth]{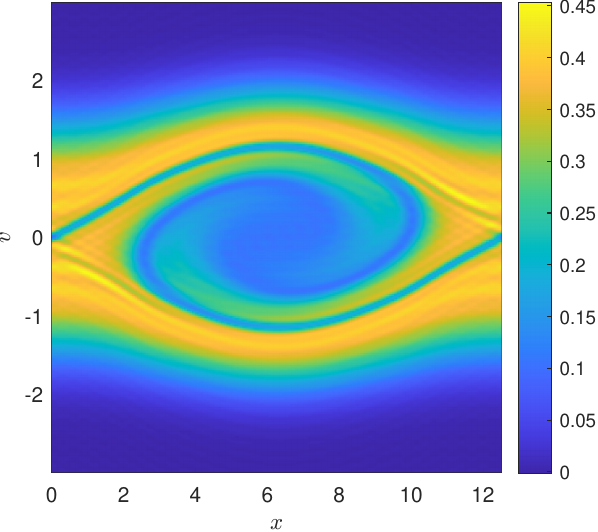}
		\caption{$N=256$, NA}
	\end{subfigure}
	\\ \medskip
    \begin{subfigure}[b]{0.4\linewidth}
		\includegraphics[width=\textwidth]{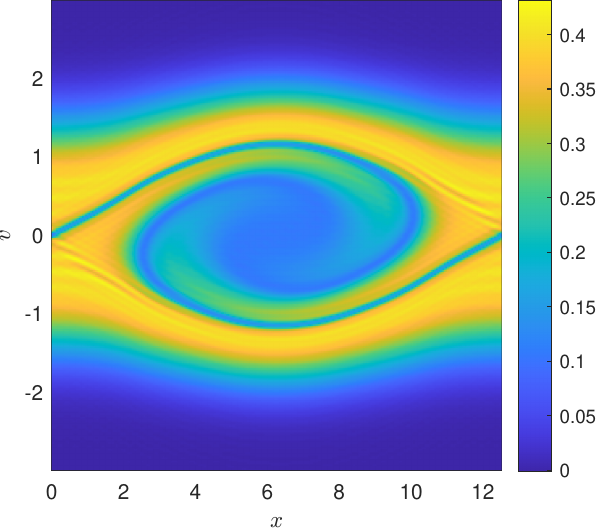}
		\caption{$N=512$, NA}
	\end{subfigure}
    \qquad
    \begin{subfigure}[b]{0.4\linewidth}
		\includegraphics[width=\textwidth]{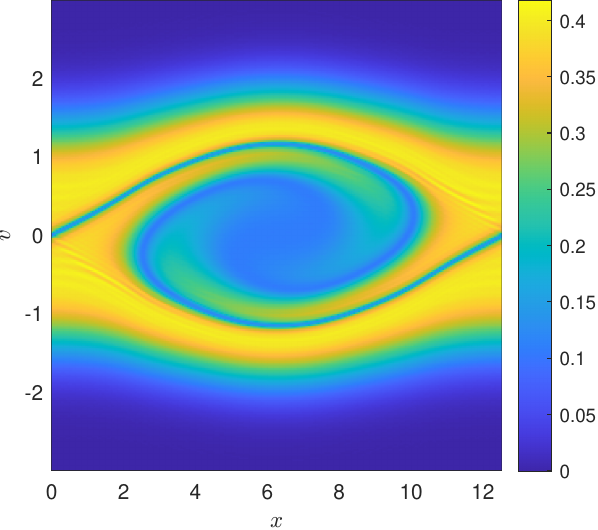}
		\caption{$N=1024$, NA}
	\end{subfigure}	
	\caption{(1D1V two-stream instability in Sec.~\ref{sec:Eg1-3}) Comparison of the distribution functions $f$ between the scaling adaptive and non-adaptive method with different expansion orders $N$, $t=50$. (a) Scaling adaptive method with $N=256$. (b--d) Non-adaptive method.}
	\label{fig:1-3-2}
\end{figure}

\begin{table}[hptb]
    \caption{(1D1V two-stream instability in Sec.~\ref{sec:Eg1-3}) Average CPU time per time step (in seconds). Here, $T_{\rm non}$ and $T_{\rm adap}$ refer to the CPU time of the non-adaptive and scaling adaptive methods, respectively. $T_{\rm ind}$ refers to the CPU time spent on the adaptive adjustment step, which is a component of $T_{\rm adap}$.}
    \label{tab:1-3-2}
    \centering
    \def\arraystretch{1.3}
    {\footnotesize
    \begin{tabular}{l|cccc}
     & $T_{\rm non} $ & $T_{\rm adap}$ & $T_{\rm ind}$ & $T_{\rm ind}/T_{\rm adap}$ \\ \hline
    $N=256$ & 1.18E$-$02 & 1.51E$-$02 & 2.82E$-$03 & 18.67\% \\
    $N=512$ & 2.36E$-$02 & 2.80E$-$02 & 3.99E$-$03 & 14.26\% \\
    $N=1024$ & 4.61E$-$02 & 5.24E$-$02 & 5.68E$-$03 & 10.83\%
    \end{tabular}
    }
\end{table}

\subsubsection{Bump-on-tail instability}\label{sec:Eg1-4}
\begin{figure}
    \centering
    \begin{subfigure}[b]{0.32\linewidth}
		\includegraphics[width=\textwidth]{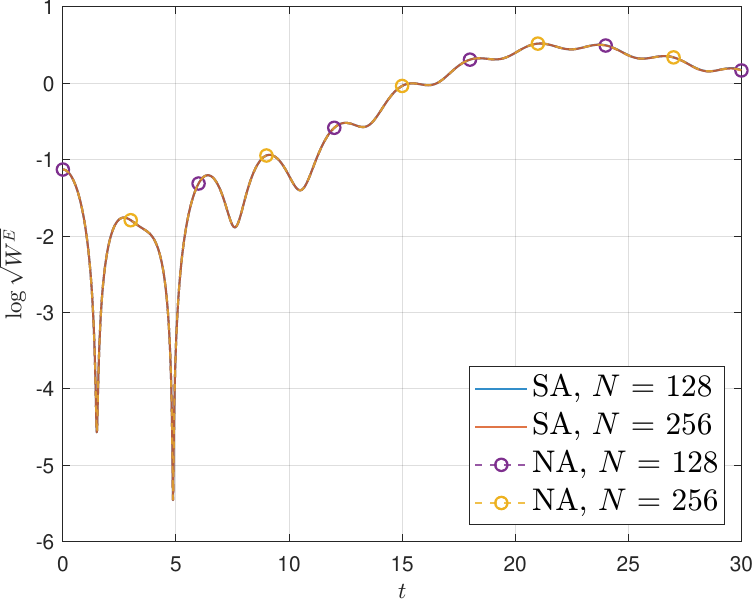}
		\caption{$W^E$}
		\label{fig:1-4-1a}
	\end{subfigure}
	\hfill
	\begin{subfigure}[b]{0.315\linewidth}
		\includegraphics[width=\textwidth]{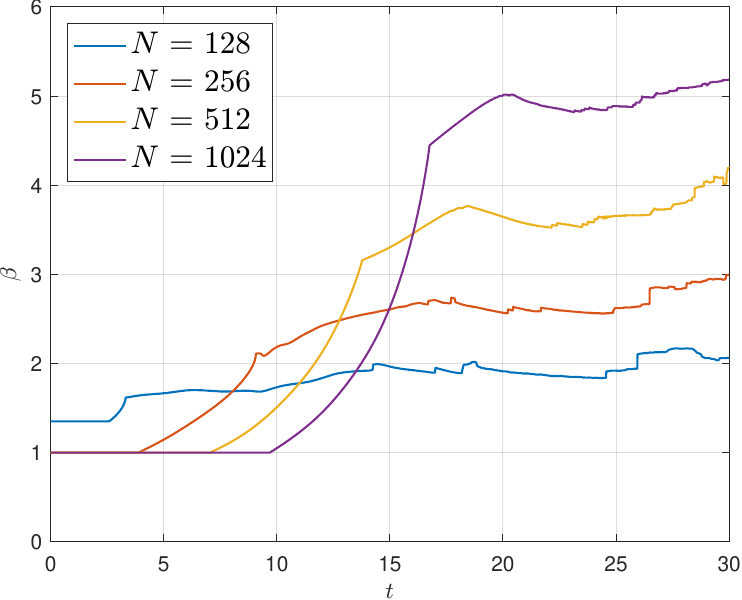}
		\caption{Scaling factor $\beta$}
		\label{fig:1-4-1b}
	\end{subfigure}
    \hfill
	\begin{subfigure}[b]{0.33\linewidth}
		\includegraphics[width=\textwidth]{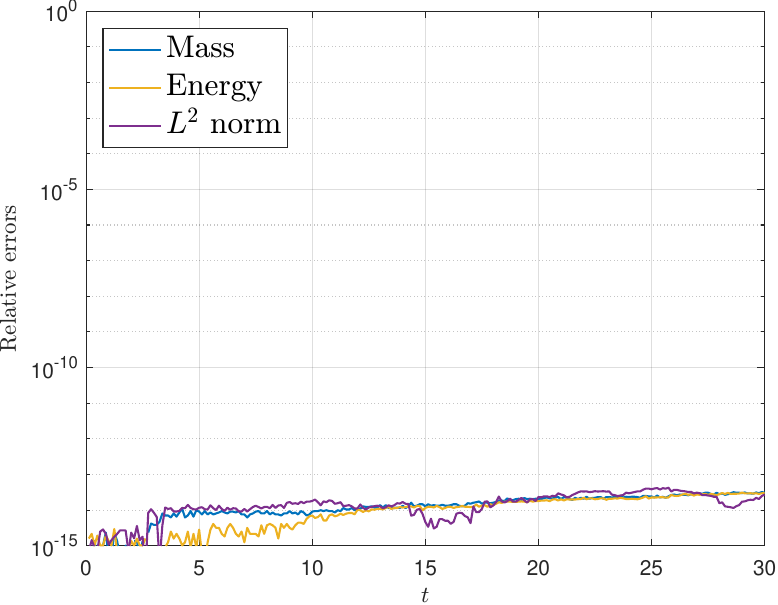}
		\caption{Errors in conserved quantities}
		\label{fig:1-4-1c}
	\end{subfigure}
	
	\caption{(1D1V bump-on-tail instability in Sec.~\ref{sec:Eg1-4}) Comparison between the scaling adaptive and non-adaptive methods with different expansion orders $N$. (a) Evolution of the potential energy $W^E$ by two methods. (b) Evolution of the scaling factor $\beta$. (c) Evolution of the errors in total mass, energy, and the $L^2$ norm, computed by the scaling adaptive method with $N=128$ and $\Delta t=0.00025$.}
	\label{fig:1-4-1}
\end{figure}

In this subsection, the bump-on-tail instability problem with the initial condition \eqref{eq:ini_bump} is considered. 
\begin{equation}
\label{eq:ini_bump}
    f_0(x,v)=\bigl(1+\alpha\cos(kx)\bigr)\left[\frac{\rho_1}{\sqrt{2\pi}}\exp\left(-\frac{v^2}{2}\right)+\frac{\rho_2}{\sqrt{2\pi\theta}}\exp\left(-\frac{(v-u)^2}{2\theta}\right)\right], \qquad x\in[0,20\pi/3],
\end{equation}
where $\rho_1=0.9$, $\rho_2=0.1$, $u=4.5$, $\theta=0.25$. $\alpha=0.03$, and $k=0.3$. The Fourier expansion order is $N_x=256$, and the CFL number is $\CFL = 0.8$. For the non-adaptive method, the scaling factor is fixed at $\beta=1$. The parameters in Tab.~\ref{tab:1-1-1} are utilized for the scaling adaptive method.

The evolution of the potential energy obtained by the scaling and non-adaptive method is shown in Fig.~\ref{fig:1-4-1a} for $N=128$ and $256$, where all the results coincide. The evolution of the scaling factor $\beta$ for $N = 128, 256, 512$, and $1024$ is presented in Fig.~\ref{fig:1-4-1b}, showing an increase over time as expected. The relative error for the mass, total energy, and the $L^2$ norm with $N=128$ is shown in Fig.~\ref{fig:1-4-1c}, where it is at the order of $10^{-14}$, indicating that the total mass, energy, and $L^2$ norm are well preserved. 

The distribution function $f$ obtained by the scaling adaptive and non-adaptive method with $N = 256, 512$, and $1024$ at $t = 30$ is presented in Fig.~\ref{fig:1-4-2}. For the non-adaptive method, the distribution function $f$ exhibits significant oscillations near $(x,v)=(15,3)$, even for $N=1024$. In contrast, the scaling adaptive method resolves a clear and sharp vortex structure for $n = 512$ and $1024$, demonstrating the improved accuracy.

The average computational time for both methods is illustrated in Tab.~\ref{tab:1-4-1}. It shows that the computational cost of the scaling adaptive method increases approximately linearly with the expansion order $N$, while the cost from adaptive adjustment remains small. These results demonstrate the efficiency gains achieved by incorporating scaling adaptivity compared with the non-adaptive approach.

\begin{figure}
    \centering
    \begin{subfigure}[b]{0.32\linewidth}
		\includegraphics[width=\textwidth]{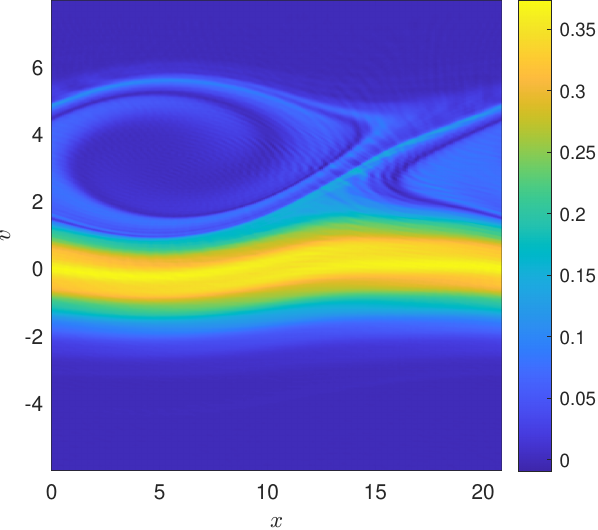}
		\caption{$N=256$, SA}
	\end{subfigure}
	\hfill
    \begin{subfigure}[b]{0.32\linewidth}
		\includegraphics[width=\textwidth]{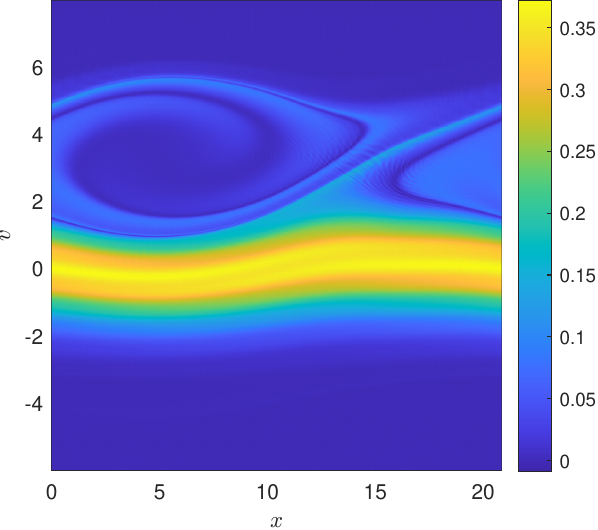}
		\caption{$N=512$, SA}
	\end{subfigure}
    \hfill
    \begin{subfigure}[b]{0.32\linewidth}
		\includegraphics[width=\textwidth]{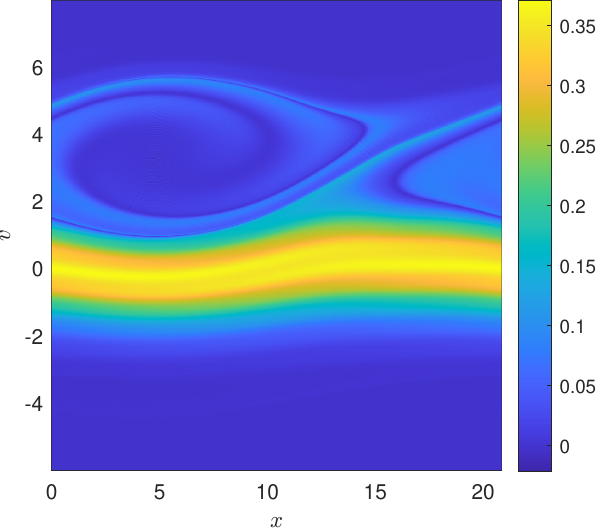}
		\caption{$N=1024$, SA}
	\end{subfigure}
	\\ \medskip
    \begin{subfigure}[b]{0.32\linewidth}
		\includegraphics[width=\textwidth]{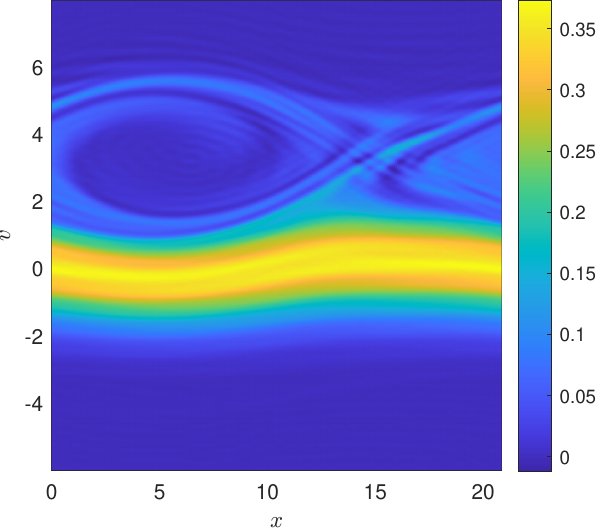}
		\caption{$N=256$, NA}
	\end{subfigure}
	\hfill
    \begin{subfigure}[b]{0.32\linewidth}
		\includegraphics[width=\textwidth]{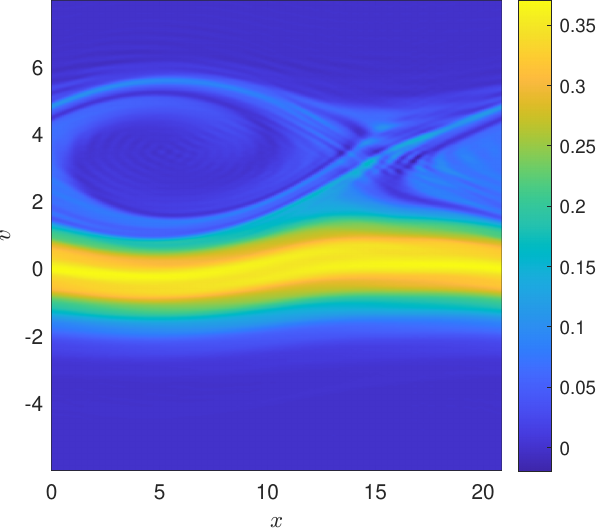}
		\caption{$N=512$, NA}
	\end{subfigure}
    \hfill
    \begin{subfigure}[b]{0.32\linewidth}
		\includegraphics[width=\textwidth]{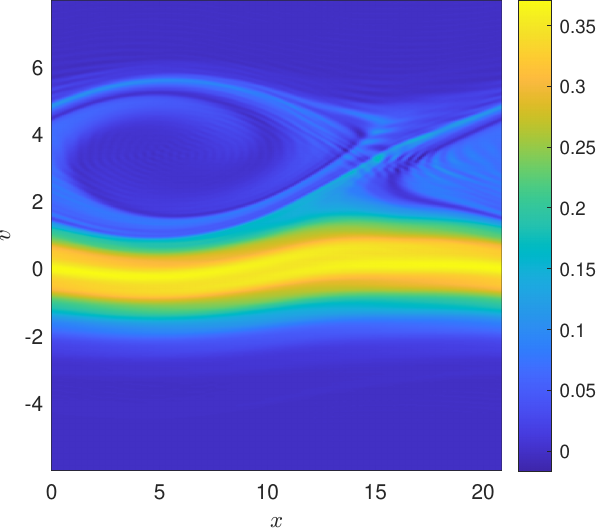}
		\caption{$N=1024$, NA}
	\end{subfigure}
	
	\caption{(1D1V bump-on-tail instability in Sec.~\ref{sec:Eg1-4}) Comparison of the distribution functions $f$ between the scaling adaptive and non-adaptive method with different expansion orders $N$, $t=30$. (a--c) Scaling adaptive method. (d--f) Non-adaptive method.}
	\label{fig:1-4-2}
\end{figure}

\begin{table}[hptb]
    \caption{(1D1V bump-on-tail instability in Sec.~\ref{sec:Eg1-4}) Average CPU time per time step (in seconds). Here, $T_{\rm non}$ and $T_{\rm adap}$ refer to the CPU time of the non-adaptive and scaling adaptive methods, respectively. $T_{\rm ind}$ refers to the CPU time spent on the adaptive adjustment step, which is a component of $T_{\rm adap}$.}
    \label{tab:1-4-1}
    \centering
    \def\arraystretch{1.3}
    {\footnotesize
    \begin{tabular}{l|cccc}
        & $T_{\rm non}$ & $T_{\rm adap}$ & $T_{\rm ind}$ & $T_{\rm ind}/T_{\rm adap}$ \\ \hline
        $N=128$ & 8.10E$-$03 & 1.11E$-$02 & 2.90E$-$03 & 26.22\% \\
        $N=256$ & 1.64E$-$02 & 2.10E$-$02 & 4.56E$-$03 & 21.73\% \\
        $N=512$ & 3.36E$-$02 & 4.24E$-$02 & 8.50E$-$03 & 20.06\% \\
        $N=1024$ & 6.39E$-$02 & 8.00E$-$02 & 1.51E$-$02 & 18.93\% \\
    \end{tabular}
    }
\end{table}

\subsection{2D2V examples}
\label{sec:2D2D}
In this subsection, three 2D2V numerical examples are studied, including the linear Landau damping, two-stream instability, and bump-on-tail instability. For the 2D2V setting, the Fourier-Hermite spectral method is extended by a direct tensor-product construction \cite{Huang2026}. Each velocity dimension $v_i$ ($i=1,2$) has its own scaling factor $\beta_i$ and truncation order $N_i$. The adaptive algorithm for the velocity space is generalized to the multi-dimensional case by defining frequency indicators for each velocity dimension and applying the one-dimensional adaptivity sequentially in each direction \cite{Xia2021a}. 
\begin{figure}
    \centering
    \begin{subfigure}[b]{0.4\linewidth}
		\includegraphics[width=\textwidth]{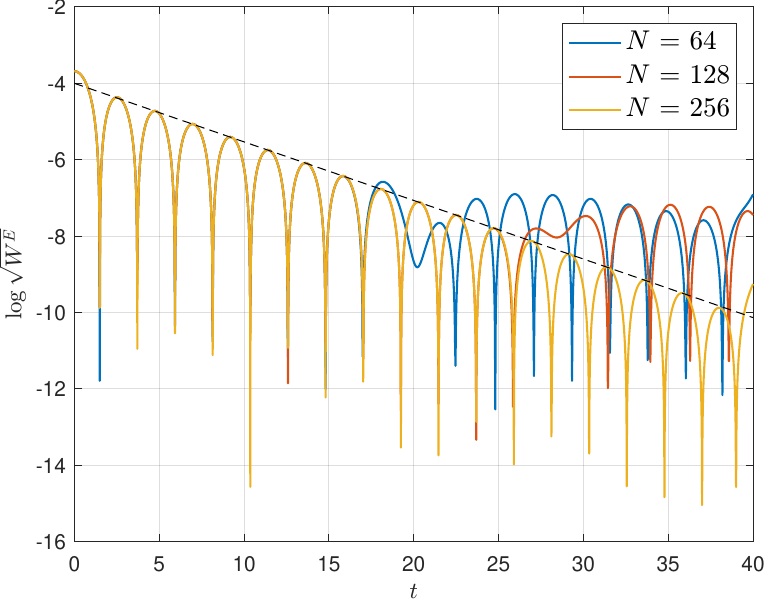}
		\caption{$W^E$, NA}
		\label{fig:2-1-1a}
	\end{subfigure}
	\qquad
	\begin{subfigure}[b]{0.4\linewidth}
		\includegraphics[width=\textwidth]{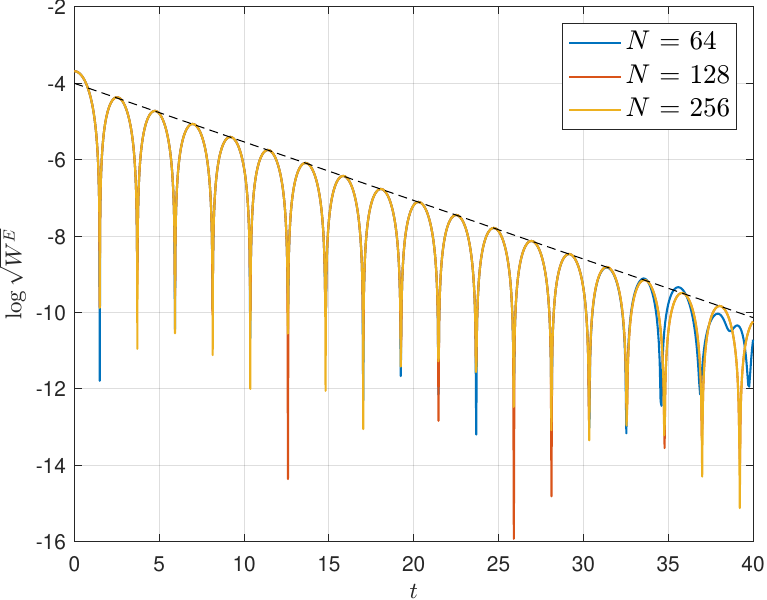}
		\caption{$W^E$, SA}
		\label{fig:2-1-1b}
	\end{subfigure}	
	\caption{(2D2V linear Landau damping in Sec.~\ref{sec:Eg2-1}) Comparison between the scaling adaptive and non-adaptive methods with different expansion orders $N$. (a) Evolution of the potential energy $W^E$ by the non-adaptive method. (b) Evolution of the potential energy $W^E$ by the scaling method.}
	\label{fig:2-1-1}
\end{figure}

\subsubsection{Linear Landau damping}\label{sec:Eg2-1}
In this subsection, the linear Landau damping problem is extended to the 2D2V case with the initial condition as follows
\begin{equation}\label{eq:exp2-Landau}
    f_0(\bm{x},\bm{v}) = \frac{1+\alpha\cos(kx_1)+\alpha\cos(kx_2)}{2\pi} \exp\left(-\frac{|\bm{v}|^2}{2}\right), \qquad \bm{x}\in[0,4\pi]^2,
\end{equation}
where $\alpha=0.001$, and $k=0.5$. The Fourier expansion order is $N_x=N_y=20$ in the spatial space, and the CFL number is set to $\CFL=0.8$. For the non-adaptive method, the scaling factors are fixed as $\beta_1=\beta_2=1$, while the parameters for the scaling adaptive method are chosen as in Tab.~\ref{tab:1-1-1}.

The time evolution of the potential energy obtained by the non-adaptive and scaling adaptive methods is shown in Fig.~\ref{fig:2-1-1a} and Fig.~\ref{fig:2-1-1b}, respectively. Similar to the one-dimensional case in Sec.~\ref{sec:Eg1-1}, the non-adaptive method exhibits recurrence phenomena. To accurately capture the decay up to $t=40$, a high expansion order of $N_1=N_2=256$ is required. In contrast, the scaling adaptive method achieves comparable accuracy with $N_1=N_2=128$ throughout the simulation.

The evolution of the scaling factors $\beta_1$ and $\beta_2$ is also examined. Owing to the isotropy of the problem, $\beta_1$ and $\beta_2$ evolve identically, which is presented in Fig.~\ref{fig:2-1-2a}. The scaling factor increases gradually over time, consistent with the behavior observed in the one-dimensional case (see Fig.~\ref{fig:1-1-2a}). The relative errors of total mass, energy, and the $L^2$ norm are shown in Fig.~\ref{fig:2-1-2b} for the scaling adaptive method with $N_1=N_2=64$. All errors remain below $10^{-12}$, demonstrating that the proposed projection algorithm preserves conservation properties with high accuracy in the high-dimensional setting.

The average computational time per time step for both methods is reported in Tab.~\ref{tab:2-1-1}. Here, $T_{\rm non}$ denotes the non-adaptive method with $N_1=N_2=256$, while $T_{\rm adap}$ corresponds to the scaling adaptive method with $N_1=N_2=128$, where comparable accuracy is achieved. It can be observed that the scaling adjustment accounts for approximately $40\%$--$50\%$ of the total computational cost of the adaptive method. Despite this additional overhead, the total time $T_{\rm adap}$ remains significantly smaller than $T_{\rm non}$. This demonstrates that the adaptive method achieves similar accuracy at roughly half the computational cost. For 2D2V problems, doubling the velocity expansion orders $N_1$ and $N_2$ leads to an approximately fourfold increase in computational cost per step. Therefore, improving resolution through scaling adaptivity rather than increasing the expansion order is a more efficient strategy.

\begin{figure}
    \centering
	\begin{subfigure}[b]{0.4\linewidth}
		\includegraphics[width=\textwidth]{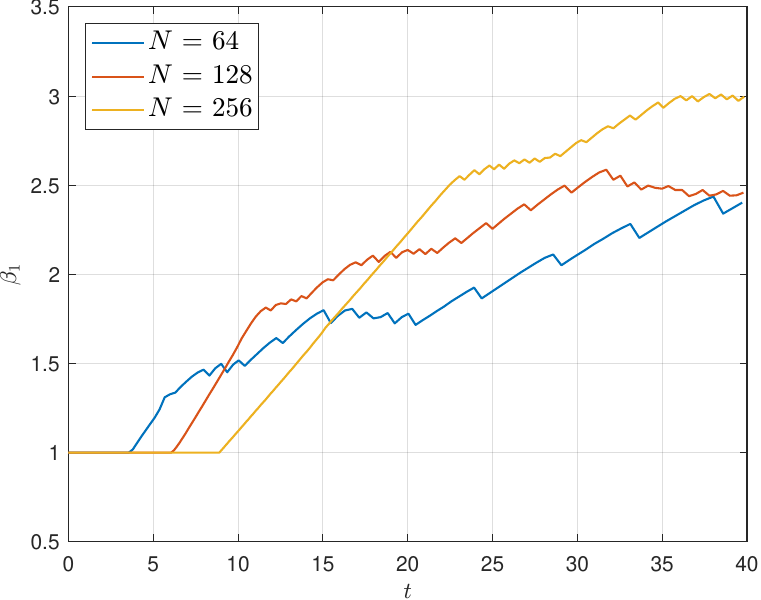}
		\caption{Scaling factor $\beta_1 = \beta_2$}
		\label{fig:2-1-2a}
	\end{subfigure}
    \qquad
    \begin{subfigure}[b]{0.41\linewidth}
		\includegraphics[width=\textwidth]{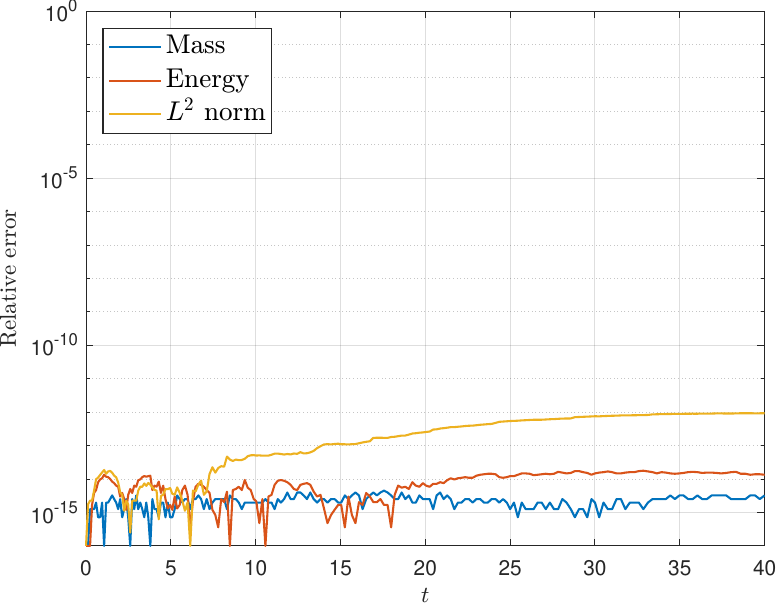}
		\caption{Errors in conserved quantities}
		\label{fig:2-1-2b}
	\end{subfigure}
	\caption{(2D2V linear Landau damping in Sec.~\ref{sec:Eg2-1}) Results of the scaling adaptive method with different expansion orders $N$. (a) Evolution of the scaling factor $\beta$. (b) Evolution of the errors in total mass, energy, and the $L^2$ norm, computed by the scaling adaptive method with $N_1=N_2=64$.}
	\label{fig:2-1-2}
\end{figure}

\begin{table}[hptb]
    \caption{(2D2V problems) Average wall-clock time (in seconds) per time step. Here, $T_{\rm non}$ and $T_{\rm adap}$ refer to the computational time of the non-adaptive and the adaptive method. $T_{\rm ind}$ refers to the computational time of the adjustment step in $T_{\rm adap}$. All tests are run with 64 threads.}
    \label{tab:2-1-1}
    \centering
    \def\arraystretch{1.3}
    {\footnotesize
    \begin{tabular}{c||rccc}
     & \multicolumn{1}{c}{$T_{\rm non}$} & $T_{\rm adap}$ & $T_{\rm ind}$ & $T_{\rm ind}/T_{\rm adap}$ \\ \hline
    Sec.~\ref{sec:Eg2-1} & 0.897 & 0.423 & 0.203 & 48.02\% \\
    Sec.~\ref{sec:Eg2-2} & 7.949 & 4.730 & 1.884 & 39.84\% \\
    Sec.~\ref{sec:Eg2-3} & 13.697 & 8.285 & 3.377 & 40.77\%
    \end{tabular}
    }
\end{table}

\subsubsection{Two-stream instability}\label{sec:Eg2-2}
In this subsection, the 2D2V two-stream instability problem is studied, which has the following initial condition 
\begin{gather*}
	f_0(\bm{x},\bm{v}) = \bigl(1+\alpha\cos(kx_1)+\alpha\cos(kx_2)\bigr) f_1(v_1) f_2(v_2),\qquad \bm{x}\in[0,40\pi/3]^2, \\
	f_1(v_1) = \sum_{i=1}^{2} \frac{\rho}{\sqrt{2\pi\theta_i}}\exp\left(-\frac{|v_1-u_i|^2}{2\theta_i}\right),\qquad 
	f_2(v_2) =\sum_{i=3}^{4} \frac{\rho}{\sqrt{2\pi\theta_i}}\exp\left(-\frac{|v_2-u_i|^2}{2\theta_i}\right),
\end{gather*}
where
\begin{equation*}
\begin{aligned}
	\rho=0.5,\qquad 
	u_1=-u_2=2.5,\quad u_3=-u_4=3.5,\qquad 
	\theta_1=\theta_2=1,\quad \theta_3=\theta_4=2,
\end{aligned}
\end{equation*}
with $\alpha=0.05$, and $k=0.15$. The spatial expansion orders are set to $N_x=N_y=96$, and the CFL number is $\CFL=0.8$. This problem is anisotropic in the microscopic velocity space, with different temperatures in the $v_1$ and $v_2$ directions. It is therefore utilized to validate the capability of the scaling adaptive method to resolve anisotropic features. The parameters of the scaling adaptive method are the same as those in Tab.~\ref{tab:1-1-1}. For the non-adaptive method, the scaling factors are fixed as $\beta_1=1/\sqrt{\theta_1}=1$ and $\beta_2=1/\sqrt{\theta_3}=1/\sqrt{2}$.

The evolution of the potential energy is shown in Fig.~\ref{fig:2-2-1a}. In this case, the non-adaptive method with a moderate expansion order $N_1=N_2=96$ already provides an accurate result, and no significant improvement from the adaptive method is observed for this quantity. The evolution of the scaling factors is presented in Fig.~\ref{fig:2-2-1b}. During the simulation, both $\beta_1$ and $\beta_2$ increase, while maintaining $\beta_2 < \beta_1$. This behavior reflects the anisotropic spreading of the distribution function and demonstrates that the adaptive strategy adjusts each velocity dimension consistently with its own features. 

The evolution of the relative error for the total mass, energy, and $L^2$ norm obtained by $N_1=N_2=32$ is shown in Fig.~\ref{fig:2-2-1c}. All errors remain at the level of machine precision, confirming that this adaptive method preserves conservation properties in the 2D2V setting.

To further examine the phase-space resolution, the distribution function is compared in the $(v_1,x)$ plane at $(v_2,y)=(0,20\pi/3)$ and in the $(v_2,y)$ plane at $(v_1,x)=(0,20\pi/3)$, as shown in Fig.~\ref{fig:2-2-2}. In the $(v_1,x)$ plane, the non-adaptive method with $N_1=N_2=96$ exhibits non-physical oscillations near the center, which are reduced when the expansion order is increased to $N_1=N_2=160$. In contrast, the scaling adaptive method produces a well-resolved structure with $N_1=N_2=96$. In the $(v_2,y)$ plane, the solution develops broader structures due to the higher temperature. Even with $N_1=N_2=160$, the non-adaptive method still exhibits noticeable oscillations, whereas the scaling adaptive method with $N_1=N_2=96$ provides a more accurate representation. These results demonstrate that scaling adaptivity effectively captures anisotropic features in multi-dimensional velocity space.

The average computational time per time step for the non-adaptive method with $N_1=N_2=160$ and the scaling adaptive method with $N_1=N_2=96$ is reported in Tab.~\ref{tab:2-1-1}. The adaptive adjustment accounts for less than $40\%$ of the total cost, and the overall computational time remains significantly lower than that of the non-adaptive method, demonstrating the efficiency of the proposed method.

\begin{figure}
    \centering
    \begin{subfigure}[b]{0.32\linewidth}
		\includegraphics[width=\textwidth]{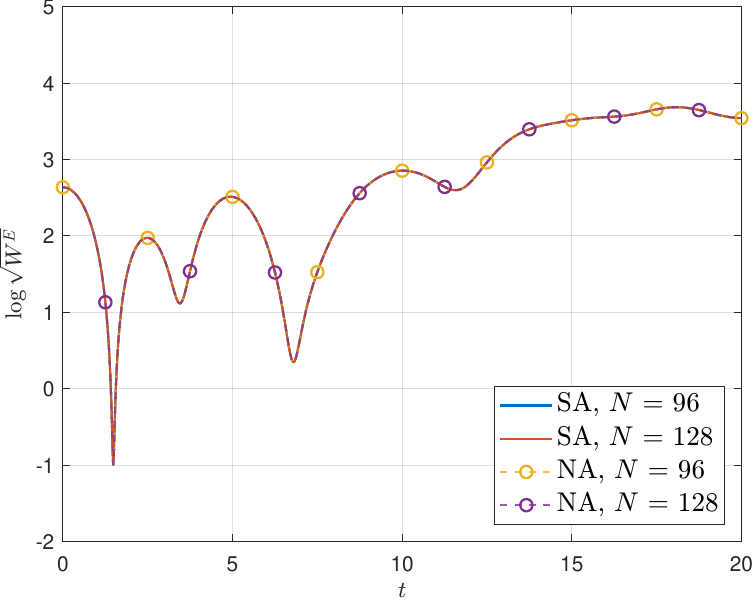}
		\caption{$W^E$}
		\label{fig:2-2-1a}
	\end{subfigure}
	\hfill
	\begin{subfigure}[b]{0.323\linewidth}
		\includegraphics[width=\textwidth]{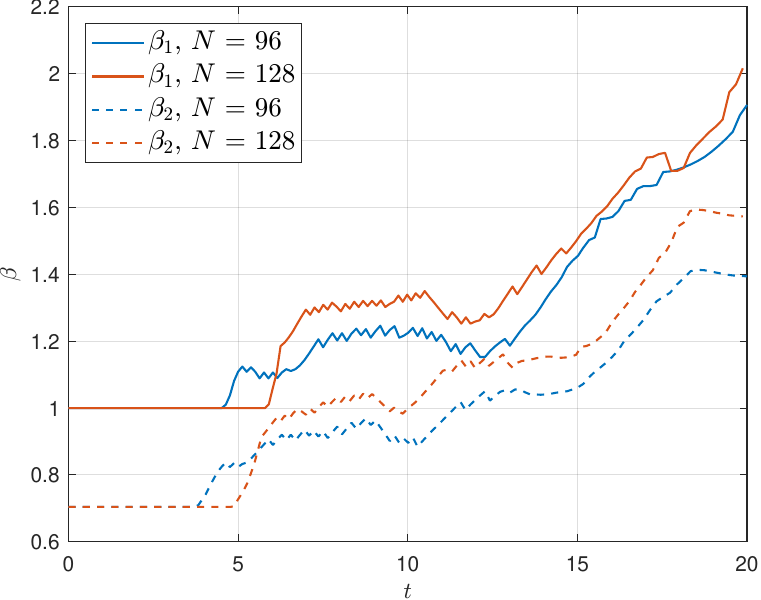}
		\caption{Scaling factors}
		\label{fig:2-2-1b}
	\end{subfigure}
    \hfill
	\begin{subfigure}[b]{0.33\linewidth}
		\includegraphics[width=\textwidth]{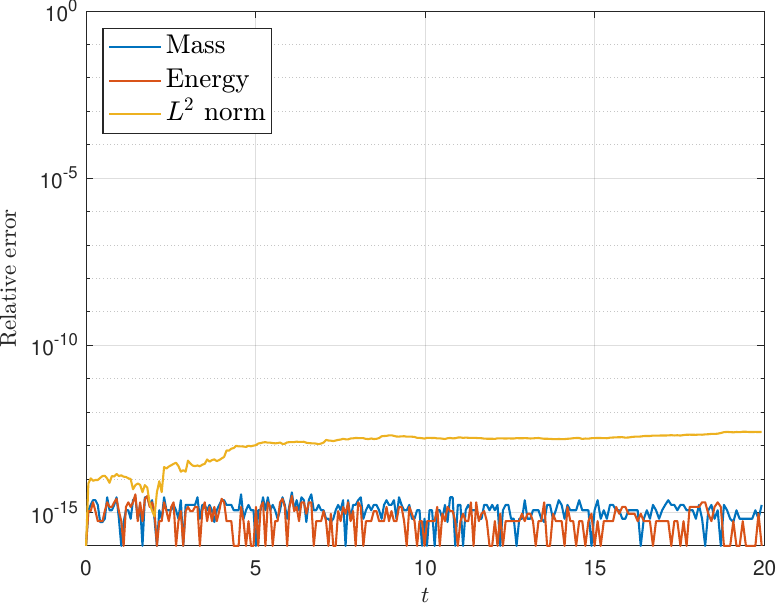}
		\caption{Errors in conserved quantities}
		\label{fig:2-2-1c}
	\end{subfigure}
	
	\caption{(2D2V two-stream instability in Sec.~\ref{sec:Eg2-2}) Comparison between the scaling adaptive and non-adaptive methods with different expansion orders $N$. (a) Evolution of the potential energy $W^E$ by two methods. (b) Evolution of the scaling factor $\beta$. (c) Evolution of the relative errors in total mass, energy, and the $L^2$ norm, computed by the scaling adaptive method with $N_1=N_2=32$.}
	\label{fig:2-2-1}
\end{figure}

\begin{figure}
    \centering
    \begin{subfigure}[b]{0.24\linewidth}
		\includegraphics[width=\textwidth]{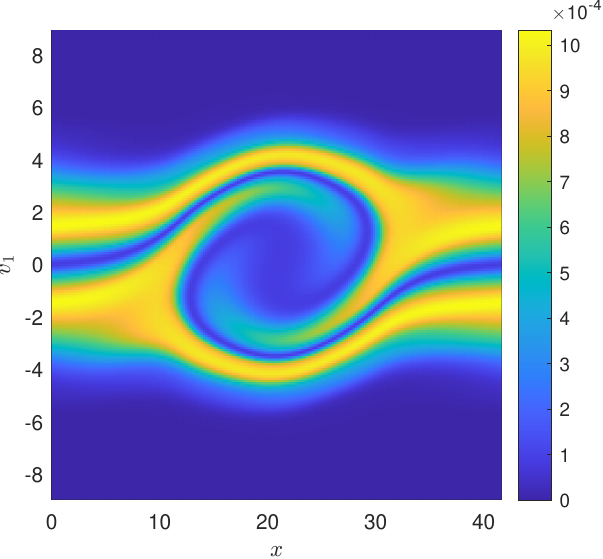}
		\caption{$(v_1,x)$, $N=96$, SA}
	\end{subfigure}
    \hfill
    \begin{subfigure}[b]{0.24\linewidth}
		\includegraphics[width=\textwidth]{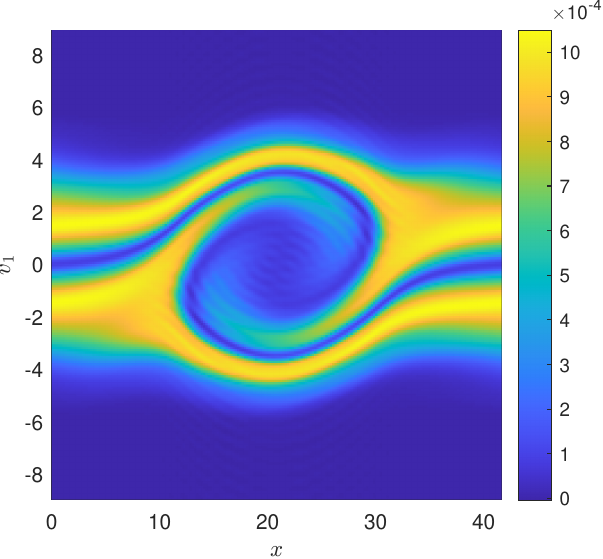}
		\caption{$(v_1,x)$, $N=96$, NA}
	\end{subfigure}
	\hfill
    \begin{subfigure}[b]{0.24\linewidth}
		\includegraphics[width=\textwidth]{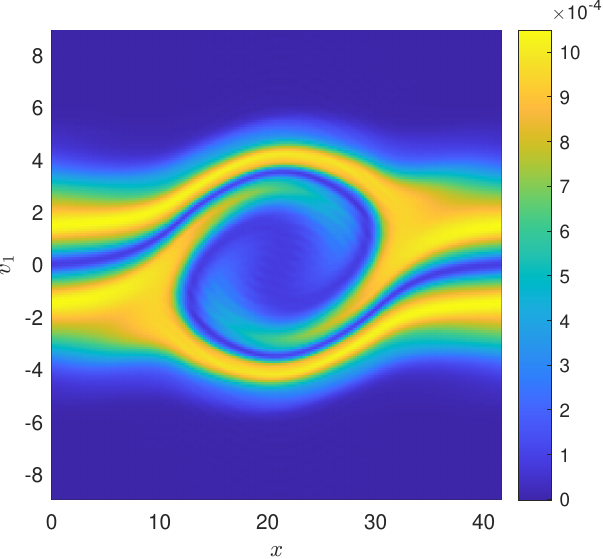}
		\caption{$(v_1,x)$, $N=128$, NA}
	\end{subfigure}
    \hfill
    \begin{subfigure}[b]{0.24\linewidth}
		\includegraphics[width=\textwidth]{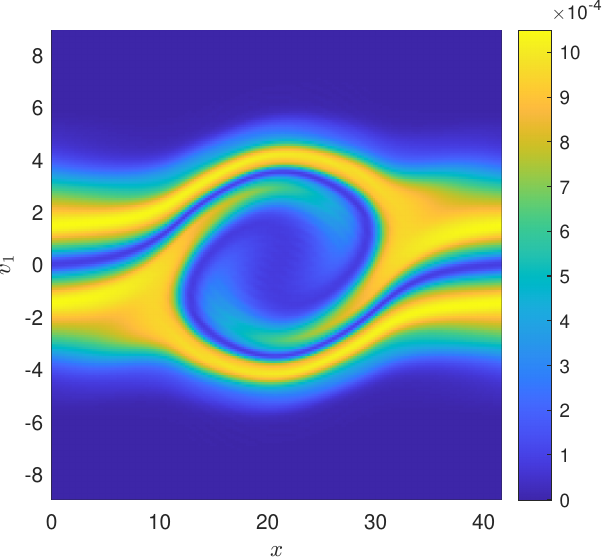}
		\caption{$(v_1,x)$, $N=160$, NA}
	\end{subfigure}
    \\ \medskip
    \begin{subfigure}[b]{0.24\linewidth}
		\includegraphics[width=\textwidth]{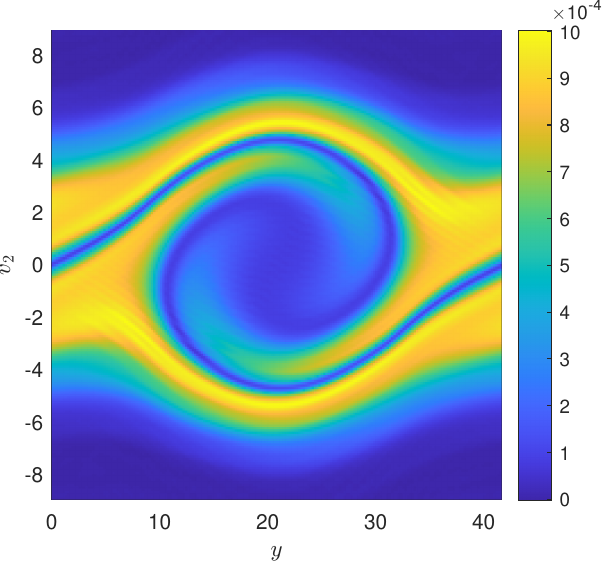}
		\caption{$(v_2,y)$, $N=96$, SA}
	\end{subfigure}
    \hfill
    \begin{subfigure}[b]{0.24\linewidth}
		\includegraphics[width=\textwidth]{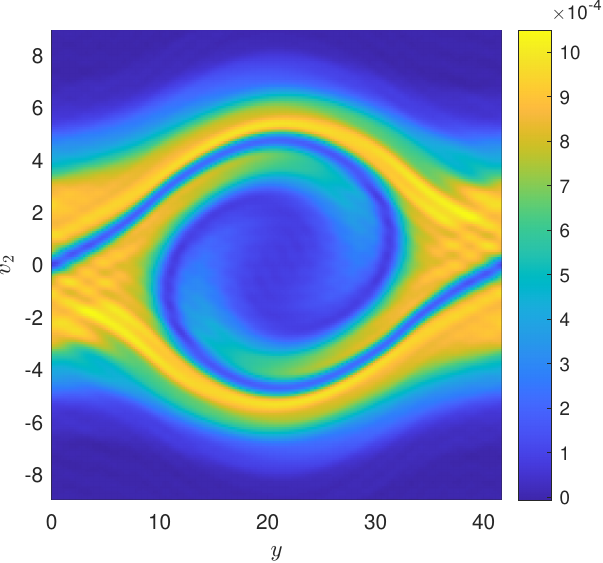}
		\caption{$(v_2,y)$, $N=96$, NA}
	\end{subfigure}
	\hfill
    \begin{subfigure}[b]{0.24\linewidth}
		\includegraphics[width=\textwidth]{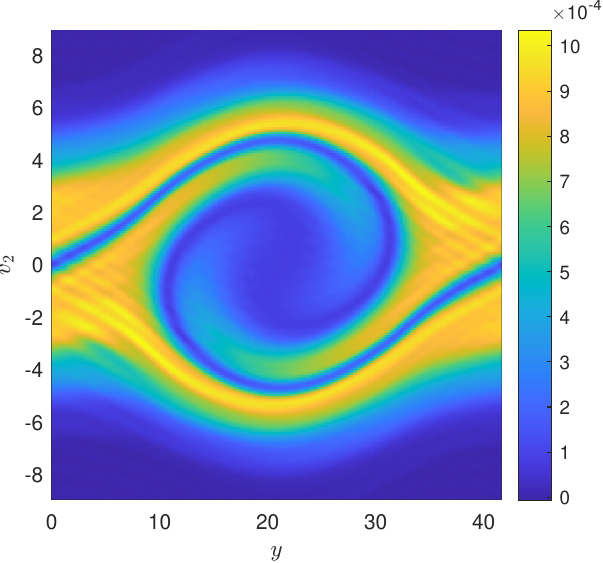}
		\caption{$(v_2,y)$, $N=128$, NA}
	\end{subfigure}
    \hfill
    \begin{subfigure}[b]{0.24\linewidth}
		\includegraphics[width=\textwidth]{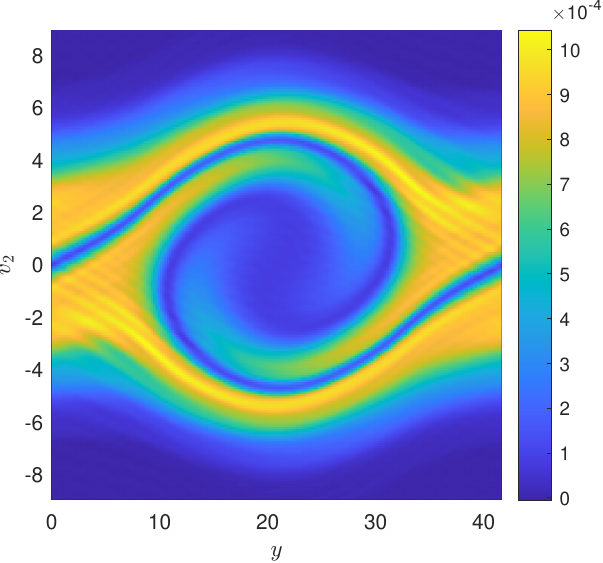}
		\caption{$(v_2,y)$, $N=160$, NA}
	\end{subfigure}
    
	\caption{(2D2V two-stream instability in Sec.~\ref{sec:Eg2-2}) Comparison of the distribution $f$ between the scaling adaptive and non-adaptive method with different expansion orders $N$ at $t=20$. (a--d) Distribution $f(v_1,0,x,L/2)$. (e--h) Distribution $f(0,v_2,L/2,y)$.}
	\label{fig:2-2-2}
\end{figure}

\subsubsection{Bump-on-tail instability}\label{sec:Eg2-3}
The 2D2V bump-on-tail instability problem is considered here, with the initial condition as 
\begin{gather*}
	f_0(\bm{x},\bm{v}) = \bigl(1+\alpha\cos(kx_1)+\alpha\cos(kx_2)\bigr) f_1(v_1) f_2(v_2),\qquad \bm{x}\in[0,20\pi/3]^2,\\
	f_1(v_1) = \sum_{i=1}^{2} \frac{\rho_i}{\sqrt{2\pi\theta_i}}\exp\left(-\frac{|v_1-u_i|^2}{2\theta_i}\right),\qquad 
	f_2(v_2) =\sum_{i=3}^{5} \frac{\rho_i}{\sqrt{2\pi\theta_i}}\exp\left(-\frac{|v_2-u_i|^2}{2\theta_i}\right),
\end{gather*}
where
\begin{equation*}
\begin{gathered}
	\rho_1=0.9,\quad \rho_2=\rho_4=0.1,\qquad \rho_3=0.85,\qquad \rho_5=0.05,\\
	u_1=u_3=0,\quad u_2=u_4=-u_5=4.5,\qquad 
	\theta_1=\theta_3=1,\quad \theta_2=\theta_4=\theta_5=0.25,
\end{gathered}
\end{equation*}
with $\alpha=0.05$, $k=0.3$. The Fourier expansion orders are set to $N_x=N_y=64$, and the CFL number is $\CFL=0.8$. This initial distribution is weakly anisotropic, with an additional high-velocity component in the $v_2$ direction. The scaling factors in the non-adaptive method are fixed as $\beta_1=\beta_2=1$, while the parameters for the scaling adaptive method are given in Tab.~\ref{tab:1-1-1}.

The evolution of the potential energy is shown in Fig.~\ref{fig:2-3-1a}. The results produced by the adaptive and non-adaptive methods are nearly indistinguishable, indicating that both approaches capture the global dynamics accurately at moderate resolution. The evolution of the scaling factors is presented in Fig.~\ref{fig:2-3-1b}. Both $\beta_1$ and $\beta_2$ increase over time. A slight separation between $\beta_1$ and $\beta_2$ is observed, consistent with the weak anisotropy of the initial condition. Fig.~\ref{fig:2-3-1c} shows the relative errors in total mass, energy, and the $L^2$ norm. All quantities are preserved at the level of machine precision throughout the simulation, confirming the conservation properties of the proposed method.

The distribution function is further compared in Fig.~\ref{fig:2-3-2}. In the $(v_1,x)$ plane, the non-adaptive method with $N_1=N_2=192$ exhibits pronounced non-physical oscillations, which remain noticeable even when the resolution is increased to $N_1=N_2=320$. In contrast, the scaling adaptive method produces a smooth and well-resolved structure already with $N_1=N_2=192$. Similar behavior is observed in the $(v_2,y)$ plane, where the adaptive method consistently captures finer structures more accurately. These results indicate that the scaling adaptive method effectively improves phase-space accuracy without increasing the expansion order.

The average computational time per time step for the non-adaptive method with $N_1=N_2=320$ and the scaling adaptive method with $N_1=N_2=192$ is reported in Tab.~\ref{tab:2-1-1}. The scaling adjustments take up approximately $40\%$ of the total cost, while the overall runtime remains significantly lower than that of the non-adaptive method, demonstrating the efficiency of the adaptive strategy.

\begin{figure}
    \centering
    \begin{subfigure}[b]{0.32\linewidth}
		\includegraphics[width=\textwidth]{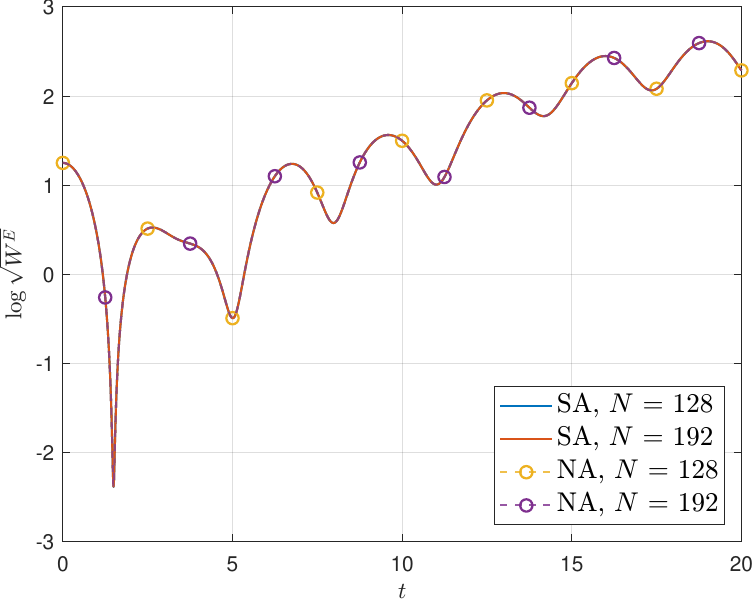}
		\caption{$W^E$}
		\label{fig:2-3-1a}
	\end{subfigure}
	\hfill
	\begin{subfigure}[b]{0.322\linewidth}
		\includegraphics[width=\textwidth]{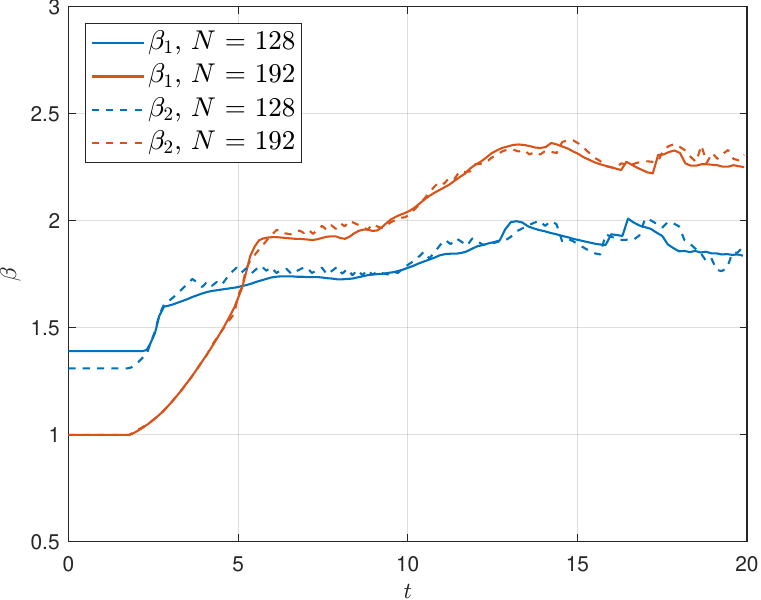}
		\caption{Scaling factors}
		\label{fig:2-3-1b}
	\end{subfigure}
    \hfill
	\begin{subfigure}[b]{0.329\linewidth}
		\includegraphics[width=\textwidth]{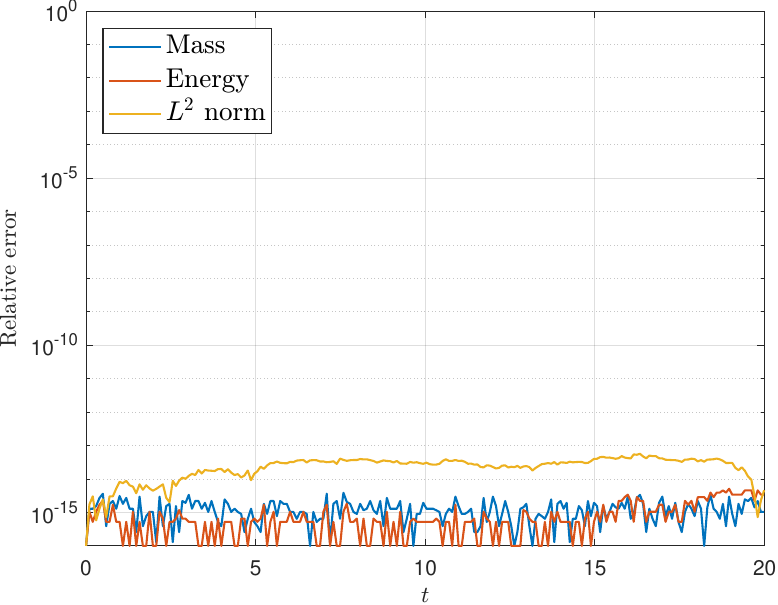}
		\caption{Errors in conserved quantities}
		\label{fig:2-3-1c}
	\end{subfigure}
	
	\caption{(2D2V bump-on-tail instability in Sec.~\ref{sec:Eg2-3}) Comparison between the scaling adaptive and non-adaptive methods with different expansion orders $N$. (a) Evolution of the potential energy $W^E$ by two methods. (b) Evolution of the scaling factor $\beta$. (c) Evolution of the relative errors in total mass, energy, and the $L^2$ norm, computed by the scaling adaptive method with $N_1=N_2=32$.}
	\label{fig:2-3-1}
\end{figure}

\begin{figure}
    \centering
    \begin{subfigure}[b]{0.32\linewidth}
		\includegraphics[width=\textwidth]{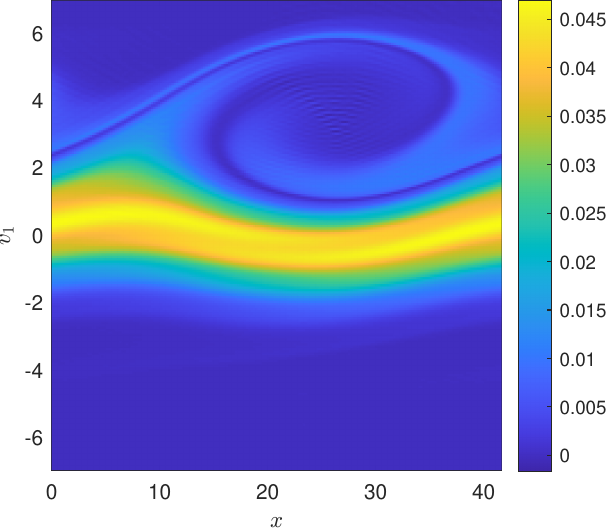}
		\caption{$(v_1,x)$, $N=192$, SA}
	\end{subfigure}
	\hfill
    \begin{subfigure}[b]{0.32\linewidth}
		\includegraphics[width=\textwidth]{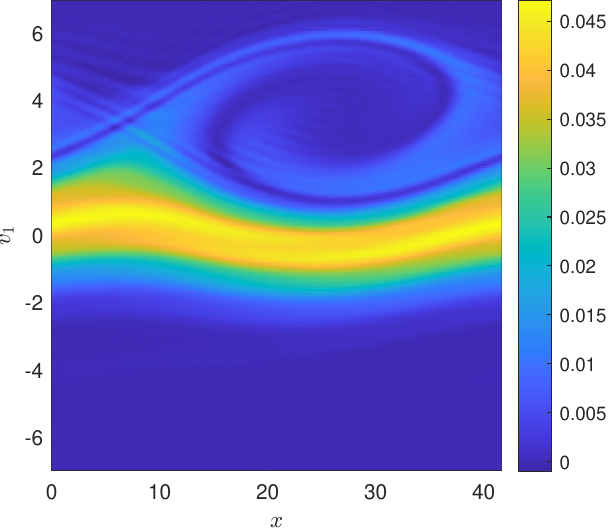}
		\caption{$(v_1,x)$, $N=192$, NA}
	\end{subfigure}
    \hfill
    \begin{subfigure}[b]{0.32\linewidth}
		\includegraphics[width=\textwidth]{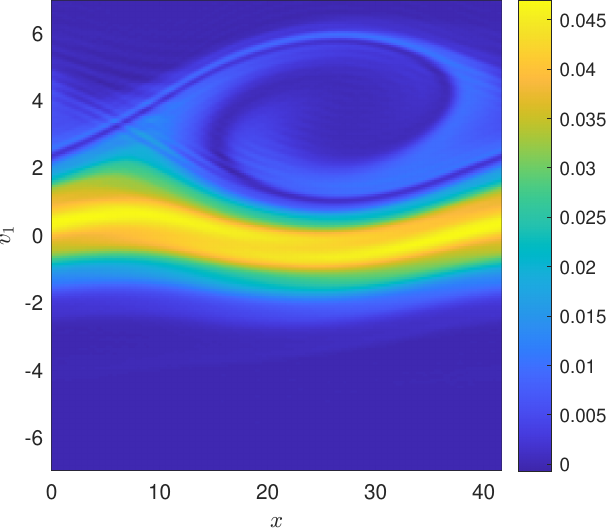}
		\caption{$(v_1,x)$, $N=320$, NA}
	\end{subfigure}
	\\ \medskip
    \begin{subfigure}[b]{0.32\linewidth}
		\includegraphics[width=\textwidth]{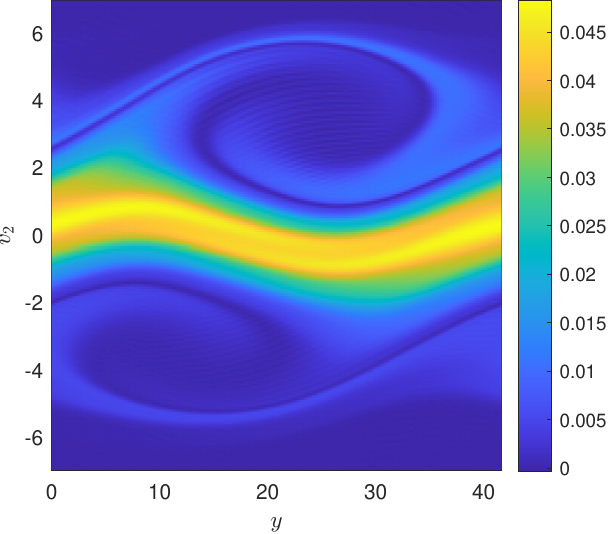}
		\caption{$(v_2,y)$, $N=192$, SA}
	\end{subfigure}
	\hfill
    \begin{subfigure}[b]{0.32\linewidth}
		\includegraphics[width=\textwidth]{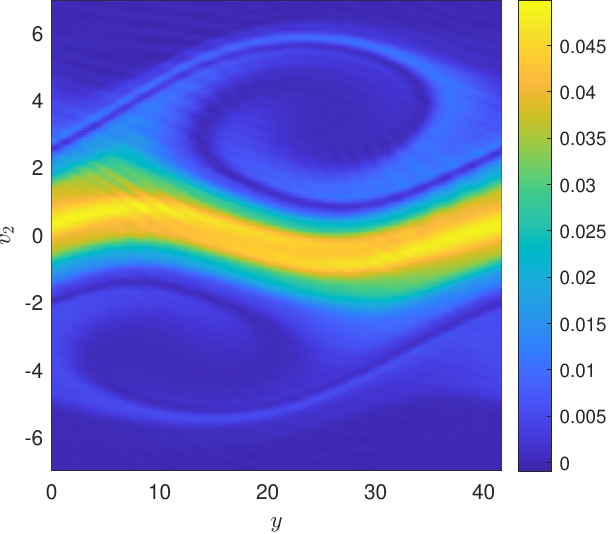}
		\caption{$(v_2,y)$, $N=192$, NA}
	\end{subfigure}
    \hfill
    \begin{subfigure}[b]{0.32\linewidth}
		\includegraphics[width=\textwidth]{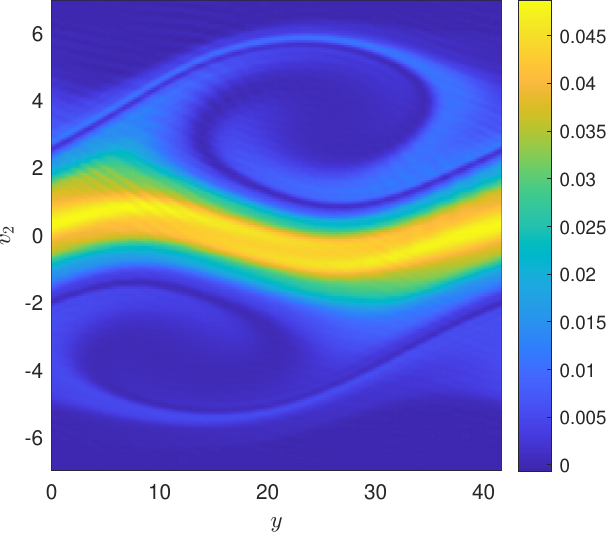}
		\caption{$(v_2,y)$, $N=320$, NA}
	\end{subfigure}
	
	\caption{(2D2V bump-on-tail instability in Sec.~\ref{sec:Eg2-3}) Comparison of the distribution $f$ between the scaling adaptive and non-adaptive method with different expansion orders $N$, $t=20$. (a--c) Distribution $f(v_1,0,x,L/2)$. (d--f) Distribution $f(0,v_2,L/2,y)$.}
	\label{fig:2-3-2}
\end{figure}

\section{Conclusions}
\label{sec:con}

In this work, an adaptive algorithm is developed in the framework of the Hermite spectral method for the Vlasov-Poisson system. The method employs symmetrically weighted Hermite basis functions and incorporates a scaling adaptive algorithm based on a frequency indicator constructed from the expansion coefficients. To enable efficient scaling adjustments, a fast conservative projection operator with linear computational complexity is proposed to ensure the conservation in both macroscopic variables and the $L^2$ norm.

The proposed method has been validated on four 1D1V and three 2D2V benchmark problems. The results demonstrate that the scaling adaptivity effectively captures the filamentation while maintaining the desired conservation properties. For future work, we aim to extend the adaptive Hermite method to the Vlasov--Maxwell system and other kinetic models.

\appendix




\section*{Acknowledgements}
The first and third authors were partially supported by the National Natural Science Foundation of China (Nos. 12325112, 12288101). This work of the second author was partially supported by the National Key Laboratory of Computational Physics, No. SYSM-2006-WDZC-15, and the Science Challenge program (NO. TZ2025016). This work was supported by the High-performance Computing Platform of Peking University. The authors would like to thank Yixiao Tang for his helpful discussions during the early stages of this work.
\section{Pseudocodes of the adaptive algorithms}\label{sec:appA}
The pseudocode of the scaling adaptive algorithm is presented in Alg.~\ref{alg:ada-s}. The complete numerical scheme is presented in Alg.~\ref{alg:compl_scheme}. The scheme that omits Step~\ref{sch:43} and uses a fixed $\beta$ corresponds to the non-adaptive Hermite method.

\begin{algorithm}
\caption{The scaling adaptive algorithm.}
\label{alg:ada-s}
\begin{algorithmic}[1]
    \Require Coefficients $\left\{\hat{f}_{k,j}^{\beta}\right\}_{k,j}$, scaling factor $\beta$, reference indicator $\mathcal{F}^{(s)}$;
    \Para $\beta_{\max}$, $\beta_{\min}$, $0<\eta_{l}^{(s)}\leqslant1\leqslant\eta_{h}^{(s)}$, $0<q<1$, $\mathcal{F}_0>0$.    
    \While{$\mathcal{F}[f_{N,N_x}^{\beta}](t^n)> \mathcal{F}_0$ \AND ($\mathcal{F}[f_{N,N_x}^{\beta}](t^n)>\eta_{h}^{(s)} \mathcal{F}^{(s)}$ \OR $\mathcal{F}[f_{N,N_x}^{\beta}](t^n)<\eta_{l}^{(s)} \mathcal{F}^{(s)}$})
    \State $\beta_{\rm de} \gets \beta q$;
    \State $f_{N,N_x}^{\beta_{\rm de}} \gets \mathcal{T}^{\beta\ra\beta_{\rm de}}f_{N,N_x}^{\beta}$;
    \State $\beta_{\rm in} \gets \beta/q$;
    \State $f_{N,N_x}^{\beta_{\rm in}} \gets \mathcal{T}^{\beta\ra\beta_{\rm in}}f_{N,N_x}^{\beta}$;
    \If{$\mathcal{F}[f_{N,N_x}^{\beta_{\rm de}}](t^n)\leqslant\min\left(\mathcal{F}[f_{N,N_x}^{\beta}](t^n),\mathcal{F}[f_{N,N_x}^{\beta_{\rm in}}](t^n)\right)$ \AND $\beta_{\rm de}\geqslant\beta_{\min}$}
    \State $\beta \gets \beta_{\rm de}$;
    \State $f_{N,N_x}^{\beta}\gets f_{N,N_x}^{\beta_{\rm de}}$;
    \ElsIf{$\mathcal{F}[f_{N,N_x}^{\beta_{\rm in}}](t^n)\leqslant\min\left(\mathcal{F}[f_{N,N_x}^{\beta_{\rm de}}](t^n),\mathcal{F}[f_{N,N_x}^{\beta}](t^n)\right)$ \AND $\beta_{\rm in}\leqslant\beta_{\max}$}
    \State $\beta \gets \beta_{\rm in}$;
    \State $f_{N,N_x}^{\beta}\gets f_{N,N_x}^{\beta_{\rm in}}$;
    \Else
    \State \Break;
    \EndIf
    \EndWhile
    \If{$\beta$ is modified}
    \State $\mathcal{F}^{(s)}\gets\mathcal{F}[f_{N,N_x}^{\beta}](t^n)$;
    \EndIf
\end{algorithmic}
\end{algorithm}


\begin{algorithm}
\caption{Adaptive Hermite method for the Vlasov-Poisson system.}
\label{alg:compl_scheme}
\begin{algorithmic}[1]
    \Require Initial distribution function $f_0(x,v)$, velocity expansion orders $N$, spatial expansion order $N_x$, time step size $\Delta t$;
    \State At the initial time step $t^0 = 0$, for the given initial distribution $f(0,x,v)$ and expansion order $N$, determine the optimal scaling factor $\beta$ by \eqref{eq:beta_init} and compute the coefficients $\left(\hat{f}_{k,j}^{\beta}\right)^0$ for $0\leqslant k\leqslant N$, $0\leqslant j<N_x$. Initialize the reference indicator $\mathcal{F}^{(s)}$.
    
    \State \label{sch:42}At time step $t^n$, solve the ODE system \eqref{eq:ODEs} using a fourth-order Runge-Kutta method (RK4) to obtain $\left(\hat{f}_{k,j}^{\beta}\right)^{n+1}$ for $0\leqslant k\leqslant N$, $0\leqslant j<N_x-1$.
    
    \State \label{sch:43}Check whether a scaling adjustment is required using Alg.~\ref{alg:ada-s}. If so, update the scaling factor $\beta$ and coefficients $\left(\hat{f}_{k,j}^{\beta}\right)^{n+1}$ for $0\leqslant k\leqslant N$, $0\leqslant j<N_x-1$.
    
    
    \State If the final time is not reached, update $t^{n+1} =t^n+\Delta t$ and go back to Step~\ref{sch:42}.
\end{algorithmic}
\end{algorithm}

\section{Analysis of the consistency condition}\label{sec:appB}
This section analyzes the consistency condition \eqref{eq:opt_consis}, which can be rewritten as the quadratic inequality
\begin{equation}\label{eq:consis_quad}
    \bm{f}^{\beta^\top}\left(I_{N+1} -B\right)\bm{f}^{\beta} \geqslant 0, \qquad     B\coloneq \mathcal{I}_{\beta}^{}\left(\mathcal{I}_{\beta'}^{\top}\mathcal{I}_{\beta'}^{}\right)^{-1}\mathcal{I}_{\beta}^{\top}.
\end{equation}
Therefore, the feasibility of the consistency condition is determined by the spectrum of $I_{N+1}-B$, or equivalently, the spectrum of $B$. Observing that
\begin{equation}
    \mathcal{I}_{\beta}^{} = \mathcal{I}_1 D_{\beta},
\end{equation}
where $D_{\beta}\coloneq \mathrm{diag}\left(\beta^{1/2},\beta^{3/2},\beta^{5/2}\right)\in\RR^{3\times3}$, it holds for $B$ that 
\begin{equation}
    B = \mathcal{I}_{1}^{}D_{\beta}^{-1}\left(D_{\beta'}^{-1}\mathcal{I}_{1}^{\top}\mathcal{I}_{1}^{}D_{\beta'}^{-1}\right)^{-1}D_{\beta}^{-1}\mathcal{I}_{1}^{\top} = \mathcal{I}_{1}^{}D_{q}\left(\mathcal{I}_{1}^{\top}\mathcal{I}_{1}^{}\right)^{-1}D_{q}\mathcal{I}_{1}^{\top},
\end{equation}
which only depends on the ratio of the scaling factors $q=\beta'/\beta$. Since $\mathrm{rank}(B)=3$, the matrix $B$ only has three nonzero eigenvalues, denoted by $\lambda_1,\lambda_2,\lambda_3$, which coincide the eigenvalues of the $3\times3$ matrix
\begin{equation}
    \tilde{B}\coloneq D_{q}\left(\mathcal{I}_{1}^{\top}\mathcal{I}_{1}^{}\right)^{-1}D_{q}\left(\mathcal{I}_{1}^{\top}\mathcal{I}_{1}^{}\right).
\end{equation}
Notice from \eqref{eq:I} that $\mathcal{I}_1$ exhibits a chessboard structure, namely
\begin{equation}
    I_{k,r}^1 = 0,\qquad 2\nmid k+r.
\end{equation}
Consequently, the same sparsity pattern is inherited by $\mathcal{I}_{1}^{\top}\mathcal{I}_{1}^{}$, $\left(\mathcal{I}_{1}^{\top}\mathcal{I}_{1}^{}\right)^{-1}$, and $\tilde{B}$. Assume that $\mathcal{I}_{1}^{\top}\mathcal{I}_{1}^{}$ has elements
\begin{equation}
    \mathcal{I}_{1}^{\top}\mathcal{I}_{1}^{}=\begin{pmatrix}
        a&0&b \\ 0&c&0 \\ b&0&d
    \end{pmatrix},
\end{equation}
then
\begin{equation}
    \left(\mathcal{I}_{1}^{\top}\mathcal{I}_{1}^{}\right)^{-1}=\begin{pmatrix}
        d/\Delta&0&-b/\Delta \\ 0&c^{-1}&0 \\ -b/\Delta&0&a/\Delta
    \end{pmatrix},\qquad \Delta=ad-b^2.
\end{equation}
This shows that the second coordinate decouples from the first and third in $\tilde{B}$. In particular, one can check that $(0,1,0)^\top$ is an eigenvector of $\tilde{B}$, with the eigenvalue
\begin{equation}
    \lambda_2=q^3.
\end{equation}
Therefore, when $q>1$, the matrix $\tilde{B}$ has an eigenvalue $\lambda_2>1$, implying that $I_{N+1}-B$ has a negative eigenvalue. So $I_{N+1}-B$ is not semi-definite positive for all $q>1$, and the consistency condition \eqref{eq:consis_quad} may fail in general.

Although the consistency condition may fail in principle, such violations are unlikely in practice. Precisely, for the remaining two nonzero eigenvalues of $B$, one has
\begin{align}
    \lambda_1\lambda_3 =\frac{(\det D_q)^2}{\lambda_2^2}=q^6,\qquad     \lambda_1+\lambda_3 =\mathrm{tr}(A)-\lambda_2=\frac{ad(q+q^5)-2b^2q^3}{ad-b^2}.
\end{align}
Therefore,
\begin{equation}
    \lambda_{1},\lambda_3=\frac{1}{2}\left(\frac{ad(q+q^5)-2b^2q^3}{ad-b^2}\pm\sqrt{\left(\frac{ad(q+q^5)-2b^2q^3}{ad-b^2}\right)^2-4q^6}\right).
\end{equation}
In practical applications, the scaling ratio satisfies $q=\beta'/\beta\approx1$. Suppose $q=1+\epsilon$ with $|\epsilon|\ll1$, a straightforward calculation yields
\begin{equation}
    \lambda_{1},\lambda_3=1+O(\epsilon).
\end{equation}
Hence, all three nonzero eigenvalues of $B$ remain close to $1$. Recall that the spectrum of $I_{N+1}-B$ consists of $1-\lambda_1,1-\lambda_2,1-\lambda_3$, together with the eigenvalue $1$ of multiplicity $N-2$. Therefore, 
\begin{equation}
    \sigma(I_{N+1}-B)=\{O(\epsilon),O(\epsilon),O(\epsilon),1,1,...,1\}.
\end{equation}
Even if the first three eigenvalues become slightly negative, their magnitude is of order $O(\epsilon)$. In contrast, the remaining $N-2$ eigenvalues are exactly $1$. As a result, the quadratic form \eqref{eq:consis_quad} is unlikely to become negative for typical coefficient vectors $\bm{f}^\beta$, which explains why violations of the consistency condition are not observed in numerical simulations.

\section{Proof of Prop.~\ref{prop:proj_ODE}}\label{sec:appC}
In this section, the proof of Prof.~\ref{prop:proj_ODE} is presented.

\begin{proof}[Proof of Prop.~\ref{prop:proj_ODE}]
    The three-term recurrence relation and the derivative of $\swH_k^\beta$ read
    \begin{align}
        \swH_{k+1}^\beta(v) &= \beta v\sqrt{\frac{2}{k+1}}\swH_{k}^\beta(v) - \sqrt{\frac{k}{k+1}}\swH_{k-1}^\beta(v), \label{eq:three_term}\\
        \frac{\dd\swH_k^\beta(v)}{\dd v} &= -\beta\sqrt{\frac{k+1}{2}}\swH_{k+1}^\beta(v) + \beta\sqrt{\frac{k}{2}}\swH_{k-1}^\beta(v). \label{eq:dHdv}
    \end{align}
    Using \eqref{eq:SW_basis} and \eqref{eq:dHdv}, a direct calculation gives
    \begin{align*}
        \frac{\dd\swH^\beta_k(v)}{\dd\beta}
         &= \frac{\dd}{\dd\beta}\left[\sqrt{\beta}\swH^1_k(\beta v)\right]  
         = \frac{1}{2\sqrt{\beta}}\swH^1_k(\beta v) + \sqrt{\beta}\frac{\dd}{\dd(\beta v)}\left[\swH^1_k(\beta v)\right]\frac{\dd\beta v}{\dd\beta}\\
        & = \frac{1}{2\beta}\swH^\beta_k(v) + v\left[-\sqrt{\frac{k+1}{2}}\swH^\beta_{k+1}(v) + \sqrt{\frac{k}{2}}\swH^\beta_{k-1}(v)\right].
    \end{align*}
    Further by the three-term recurrence relation \eqref{eq:three_term}, we have
    \begin{align*}
        \frac{\dd\swH^\beta_k(v)}{\dd\beta}
        = &\, \frac{1}{2\beta}\swH^\beta_k(v) - \frac{1}{\beta}\sqrt{\frac{k+1}{2}}\left[\sqrt{\frac{k+2}{2}}\swH^\beta_{k+2}(v)+\sqrt{\frac{k+1}{2}}\swH^\beta_{k}(v)\right]
         +\frac{1}{\beta}\sqrt{\frac{k}{2}}\left[\sqrt{\frac{k}{2}}\swH^\beta_{k}(v)+\sqrt{\frac{k-1}{2}}\swH^\beta_{k-2}(v)\right]\\
        = & -\frac{\sqrt{(k+1)(k+2)}}{2\beta}\swH^\beta_{k+2}(v) + \frac{\sqrt{k(k-1)}}{2\beta}\swH^\beta_{k-2}(v).
    \end{align*}
    Using this differential property, the derivative of $\tilde{h}_l(s)$ with respect to $s$ can be computed explicitly.
    \begin{align*}
        \frac{\dd\tilde{h}_l}{\dd s}
        & = \sum_{k=0}^{N} \hat{f}_k^\beta \int_{\RR}\swH_k^\beta(v)\frac{\dd\swH_l^{b}(v)}{\dd s}\,\dd v  = \sum_{k=0}^{N} \hat{f}_k^\beta \int_{\RR}\swH_k^\beta(v)\frac{\dd\swH_l^{b}(v)}{\dd b}\frac{\dd b}{\dd s}\,\dd v\\
        & = \sum_{k=0}^{N} \hat{f}_k^\beta \int_{\RR}\swH_k^\beta(v) \,b\ln\left(\frac{\beta'}{\beta}\right)  \left[-\frac{\sqrt{(l+1)(l+2)}}{2b}\swH^b_{l+2}(v) + \frac{\sqrt{l(l-1)}}{2b}\swH^b_{l-2}(v)\right]\,\dd v\\
        & = \frac{\ln(\beta'/\beta)}{2}\left(\sqrt{l(l-1)}\tilde{h}_{l-2}-\sqrt{(l+1)(l+2)}\tilde{h}_{l+2}\right).
    \end{align*}
Prop.~\ref{prop:proj_ODE} is then proved. 
\end{proof}

\bibliographystyle{abbrv}
\bibliography{sections/references}

\end{document}